\documentclass[a4paper]{amsart}

\usepackage{verbatim}
\usepackage[dvips]{color}
\usepackage{amssymb}
\usepackage[active]{srcltx}
\usepackage[colorlinks,pdfpagelabels,pdfstartview = FitH,bookmarksopen = true,bookmarksnumbered = true,linkcolor = blue,plainpages = false,hypertexnames = false,citecolor = red] {hyperref}

\allowdisplaybreaks

\title[Regularity for functionals with double phase]{Regularity for general functionals\\ with double phase}

\author[Baroni]{Paolo Baroni}
\address{Paolo Baroni\\Department of mathematical, physical and computer sciences, University of Parma\\
Parco Area delle Scienze 53/a, Campus, 43124 Parma, Italy}\email{paolo.baroni@unipr.it}

\author[Colombo]{Maria Colombo}
\address {Maria Colombo\\
Institute for Theoretical Studies, ETH Z\"urich,
\\ Clausiusstrasse 47, CH-8092 Z\"urich, Switzerland
 	}
\email{maria.colombo@eth-its.ethz.ch}

\author[Mingione]{Giuseppe Mingione}
\address{Giuseppe Mingione\\Dipartimento di Matematica e Informatica, Universit\`a di Parma\\
Parco Area delle Scienze 53/a, Campus, 43124 Parma, Italy}
\email{giuseppe.mingione@unipr.it.}

\newtheorem{theorem}{Theorem}[section]
\newtheorem{prop}{Proposition}[section]

\newtheorem{lemma}{Lemma}[section]

\theoremstyle{definition}
\newtheorem{definition}{Definition}
\newtheorem{remark}{Remark}[section]

\numberwithin{equation}{section}

\def\eqn#1$$#2$${\begin{equation}\label#1#2\end{equation}}
\def\charfn_#1{{\raise1.2pt\hbox{$\chi_{\kern-1pt\lower3pt\hbox{{$\scriptstyle#1$}}}$}}}
\newcommand{\rif}[1]{(\ref{#1})}
\newcommand{\trif}[1] {\textnormal{\rif{#1}}}

\newcommand{\stackleq}[1]{\stackrel{\rif{#1}}{ \leq}}

\newcommand{\eps}{\varepsilon}

\newcommand{\ai}{a_{{\rm i}}}

 \DeclareMathOperator*{\osc}{osc}

\def\dist{\operatorname{dist}}

\newcommand{\divo}{\textnormal{div}}

\newcommand{\data}{\texttt{data}}
\newcommand{\datao}{\texttt{data}(\Omega_0)}

\def\en{\mathbb N}
\def\er{\mathbb R}
\newcommand{\ern}{\mathbb{R}^n}

\def\loc{\operatorname{loc}}

\newcommand{\ratio}{\nu, L}

\def\mean#1{\mathchoice%
          {\mathop{\kern 0.2em\vrule width 0.6em height 0.69678ex depth -0.58065ex
                  \kern -0.8em \intop}\nolimits_{\kern -0.4em#1}}%
          {\mathop{\kern 0.1em\vrule width 0.5em height 0.69678ex depth -0.60387ex
                  \kern -0.6em \intop}\nolimits_{#1}}%
          {\mathop{\kern 0.1em\vrule width 0.5em height 0.69678ex
              depth -0.60387ex
                  \kern -0.6em \intop}\nolimits_{#1}}%
          {\mathop{\kern 0.1em\vrule width 0.5em height 0.69678ex depth -0.60387ex
                  \kern -0.6em \intop}\nolimits_{#1}}}

\def\vintslides_#1{\mathchoice%
          {\mathop{\kern 0.1em\vrule width 0.5em height 0.697ex depth -0.581ex
                  \kern -0.6em \intop}\nolimits_{\kern -0.4em#1}}%
          {\mathop{\kern 0.1em\vrule width 0.3em height 0.697ex depth -0.604ex
                  \kern -0.4em \intop}\nolimits_{#1}}%
          {\mathop{\kern 0.1em\vrule width 0.3em height 0.697ex depth -0.604ex
                  \kern -0.4em \intop}\nolimits_{#1}}%
          {\mathop{\kern 0.1em\vrule width 0.3em height 0.697ex depth -0.604ex
                  \kern -0.4em \intop}\nolimits_{#1}}}

\newcommand{\aveint}[2]{\mathchoice%
          {\mathop{\kern 0.2em\vrule width 0.6em height 0.69678ex depth -0.58065ex
                  \kern -0.8em \intop}\nolimits_{\kern -0.45em#1}^{#2}}%
          {\mathop{\kern 0.1em\vrule width 0.5em height 0.69678ex depth -0.60387ex
                  \kern -0.6em \intop}\nolimits_{#1}^{#2}}%
          {\mathop{\kern 0.1em\vrule width 0.5em height 0.69678ex depth -0.60387ex
                  \kern -0.6em \intop}\nolimits_{#1}^{#2}}%
          {\mathop{\kern 0.1em\vrule width 0.5em height 0.69678ex depth -0.60387ex
                  \kern -0.6em \intop}\nolimits_{#1}^{#2}}}
                  
\newtoks\by
\newtoks\paper
\newtoks\book
\newtoks\jour
\newtoks\yr
\newtoks\pages
\newtoks\vol
\newtoks\publ
\def\et{ \& }
\def\name[#1, #2]{#1 #2}
\def\ota{{\hbox{\bf ???}}}
\def\cLear{\by=\ota\paper=\ota\book=\ota\jour=\ota\yr=\ota
\pages=\ota\vol=\ota\publ=\ota}
\def\endpaper{\the\by, \textit{\the\paper},
{\the\jour} \textbf{\the\vol} (\the\yr), \the\pages.\cLear}
\def\endbook{\the\by, \textit{\the\book},
\the\publ, \the\yr.\cLear}
\def\endpap{\the\by, \textit{\the\paper}, \the\jour.\cLear}
\def\endproc{\the\by, \textit{\the\paper}, \the\book, \the\publ,
\the\yr, \the\pages.\cLear}

\newcommand{\dista}{\dist \,( \Omega_0, \partial\Omega)}

\begin{document}

\begin{abstract} We prove sharp regularity results for a general class of functionals of the type 
$$
w \mapsto \int F(x, w, Dw) \, dx\;,
$$
featuring non-standard growth conditions and non-uniform ellipticity properties. The model case is given by the double phase integral 
$$
w \mapsto \int b(x,w)(|Dw|^p+a(x)|Dw|^q) \, dx\;,\quad 1 <p < q\,, \quad a(x)\geq 0\;,
$$
with $0<\nu \leq b(\cdot)\leq L $. This changes its ellipticity rate according to the geometry of the level set $\{a(x)=0\}$ of the modulating coefficient $a(\cdot)$. 
We also present new methods and proofs, that are suitable to build regularity theorems for larger classes of non-autonomous functionals. Finally, we disclose some new interpolation type effects that, as we conjecture, should draw a general phenomenon in the setting of non-uniformly elliptic problems. Such effects naturally connect with the Lavrentiev phenomenon. \end{abstract}


\maketitle

\centerline {{\em To Paolo Marcellini on his 70th birthday, with admiration for}} 
 \centerline{{\em his pioneering work in the Calculus of Variations}}
\vspace{5mm}
\tableofcontents

\section{Introduction and results} 
The aim of this paper is to provide sharp and comprehensive regularity theorems for minimizers of a class of integral functionals of the Calculus of Variations exhibiting a degeneracy of double type and a strong non-uniform ellipticity. The main model case in question here is provided by the double phase functional\eqn{model}
$$
W^{1,1}(\Omega) \ni w   \mapsto\mathcal P(w, \Omega):= \int_{\Omega} (|Dw|^p+a(x)|Dw|^q) \, dx\,, $$
which is naturally defined for $w \in W^{1,1}(\Omega)$, where $\Omega \subset \er^n$ is a bounded open domain, $n \geq 2$ and it is initially assumed that $1< p \leq  q$ and $a(\cdot) \in L^{\infty}(\Omega)$. In the recent papers \cite{BCM1, CM1, CM2} the authors have provided a rather comprehensive regularity theory for local minimizers of $\mathcal P$. This culminates in establishing  the local gradient H\"older continuity of minima provided a sharp balancing condition between the closeness of $p$ and $q$ and the regularity of $a(\cdot)$ is satisfied. Notice that the H\"older gradient regularity is optimal already in the classical $p$-Laplacean case, i.e., when $a(\cdot)\equiv 0$ \cite{Uh, Ur}. More specifically, it turns out that the condition 
\eqn{bound1}
$$
a(\cdot) \in C^{0,\alpha}(\Omega)\,, \quad \alpha \in (0,1] \qquad \mbox{and}\qquad \frac qp \leq 1+\frac{\alpha}{n}
$$
is sufficient to get such maximal regularity. Actually, this has been proved in \cite{CM1} apart from the controversial borderline case $ q/p= 1+\alpha/n$, which is instead handled here for the first time; see also Remark \ref{borderline-re} below. Moreover, as proved in \cite{CM2}, assuming that minimizers are bounded allows to relax the relation between $p$ and $q$ in \rif{bound1} as follows:
\eqn{bound2}
$$
 u \in L^{\infty}(\Omega)\,, \quad  a(\cdot) \in C^{0,\alpha}(\Omega)\,, \quad \alpha \in (0,1] \qquad \mbox{and}\qquad q \leq p+\alpha\;.
$$
The counterexamples given in \cite{ELM, FMM} show that both \rif{bound1} and \rif{bound2} are sharp conditions for regularity. They do coincide for $p=n$, while \rif{bound1} naturally gets better than \rif{bound2} for $p>n$. 
Indeed, in this last case
the boundedness assumption on $u$ becomes irrelevant due to Sobolev-Morrey embedding theorem. Notice that such conditions essentially serve to contain the rate of non-uniform ellipticity of the functional $\mathcal P$, which can be measured by the distance $q-p$.  
We also remark that the first regularity results for minima of $\mathcal P$ under assumptions \rif{bound1}, namely the higher integrability and fractional differentiability of gradient, have been obtained in \cite{ELM} via the analysis of the related Lavrentiev phenomenon the functional $\mathcal P$ might exhibit. The functional appearing in \rif{model} has been first considered by Zhikov \cite{Z1, Z2, Z3} in the setting of Homogenization theory and to give new instances of the Lavrentiev phenomenon. In this respect, the main feature of the integrand 
\eqn{HH}
$$
H(x,z) := |z|^p +a(x)|z|^q=: K(x, |z|) \qquad z \in \er^n
$$
is that it changes its rate of ellipticity according to the positivity of $a(x)$. It shows $q$-growth with respect to the gradient variable $z$ on the set $\{a(x)>0\}$, and $p$-growth on $\{a(x)=0\}$. The  delicate transitions between the $p$- and the $q$-ellipticity explains the dependence of the bounds in \rif{bound1}-\rif{bound2} linking $p$ and $q$ on the H\"older exponent $\alpha$ of $a(\cdot)$. We refer to \cite{BCM2, CM1} for a more accurate description of the additional features of the integral in \rif{model} and its occurrence in applications. 

In this paper we shall deal with general functionals of the type
\eqn{funF}
$$
W^{1,1}(\Omega) \ni w   \mapsto \mathcal F(w, \Omega):=\int_{\Omega} F(x,w,Dw) \, dx \;,$$ 
where $F \colon \Omega \times \er\times \er^n\to \er$ is a Carath\'eodory integrand, initially satisfying the double-sided bound 
\eqn{growtH}
$$
\nu H(x, z) \leq F(x, v,z)  \leq LH(x, z)\;,
$$
for constants $0 < \nu \leq L$. This aims at giving an intrinsic approach to regularity of double phase functional, drawing a parallel  with the standard growth conditions of $p$-type, that is when $a(x)\equiv 0$ and therefore
$H(x, z)\equiv |z|^p$. Under growth conditions \rif{growtH} the natural notion of minimality is given as follows:
\begin{definition}\label{defimain} A function $u \in W^{1,1}_{\loc}(\Omega)$ is a local minimiser of the functional $\mathcal F$ defined in \trif{funF} if and only if $H(\cdot,Du) \in L^1(\Omega)$ and the minimality condition 
$\mathcal F(u,{\rm supp} \, (u-v)) \leq \mathcal F(v,{\rm supp} \, (u-v))$ is satisfied whenever $v \in W^{1,1}_{\rm{loc}}(\Omega)$ 
is such that ${\rm supp} \, (u-v) \subset \Omega$. In particular, a local minimizer belongs to $W^{1,p}(\Omega)$. 
\end{definition}
Integrals of the type in \rif{funF} under the growth assumptions \rif{growtH} belong to the family of functionals satisfying non-standard growth conditions of $(p,q)$-type, i.e., those whose integrands satisfy growth and coercivity conditions at different polynomial rates
\eqn{pqgen}
$$
\nu|z|^p \leq F(x,v,z) \leq L(|z|^q+1)\;, \qquad p< q\;.
$$
These functionals have been the object of intensive investigation over the last years, starting with the seminal papers of Marcellini \cite{M1, M2, M5}; see for instance \cite{BF, Breit, Choe, ELM, UU, Z2, Z3}. See also the survey \cite{Dark}. The regularity theory in the case of functionals with $(p,q)$-growth in the non-autonomous case, 
that is when the energy density $F(\cdot)$ is allowed to depend on $x$, is still a non-trivial open issue, especially when considering sharp conditions for regularity. In this respect, the growth conditions considered in \rif{growtH} represent a very significant case since the one in \rif{model} provides one of the hardest model examples of non-autonomous functionals with $(p,q)$-growth available in the literature. See for instance the lists of examples discussed in \cite{BCM2, ELM, Z1, Z2, Z3}.

In order to get higher regularity of minima, assumption \rif{growtH} are not sufficient already in the standard case $a(x)\equiv 0$. For this, after \rif{growtH}, we assume that $F(\cdot)$ is a continuous integrand, of class $C^2(\er^n\setminus\{0\})$ in the $z$-variable, and we consider the following natural assumptions:
\eqn{assF}
$$
\left\{
\begin{array}{c}
|\partial F(x,v,z)||z|+|\partial^2 F(x,v,z)||z|^2 \leq LH(x, z)\\ [8 pt]
\nu \left(|z|^{p-2}+a(x)|z|^{q-2}\right)|\xi|^2\leq  \big\langle \partial^2 F(x,v,z)\xi, \xi\big\rangle \\[8 pt]
| \partial F(x_1,v,z)-\partial F(x_2,v,z)||z|
 \leq L \omega\left(|x_1-x_2|\right)\left[H(x_1,z)+H(x_2,z)\right]\\ [4 pt]\hspace{3cm} + L|a(x_1)-a(x_2)||z|^{q}
\\[8 pt]
| F(x,v_1,z)-F(x,v_2,z)|\leq L \omega(|v_1-v_2|)H(x,z)\;.
 \end{array}\right.
 $$
These are assumed to hold whenever $x, x_1, x_2 \in \Omega$, $v, v_1, v_2 \in \er$, $z \in \er^n\setminus\{0\}$, $\xi \in \er^n$, where $0<\nu\leq  L$ are fixed constants, and
\eqn{convex-omega}
$$\omega(t):= \min\left\{t^{\beta}, 1\right\} \quad \mbox{for}\ t >0$$
is the standard $\beta$-H\"older continuous modulus of continuity for $\beta \in (0,1]$. In \rif{assF} the symbol $\partial$ stands for the partial derivative with respect to the $z$-variable. It is not difficult to see that assumptions \rif{assF} are for instance satisfied by the model functional
\eqn{model-2}
$$
W^{1,1}(\Omega) \ni w   \mapsto \int_{\Omega} b(x,w)H(x, Dw) \, dx\;, $$
where $0< \nu_1 \leq b(x, v)\leq L_1$, for some constants $\nu_1, L_1$ and for some $b(\cdot)$ H\"older continuous function. Further model cases are clearly given by integrals of the type
\eqn{model-3}
$$
w
\mapsto \int_{\Omega} \left[F_1(x, w, Dw) + a(x)F_{2}(x, w, Dw)\right] \, dx\;,
$$
where $F_1(\cdot)$ and $F_2(\cdot)$ have $p$- and $q$-growth, respectively, and satisfy conditions suited to imply H\"older continuity of the gradient of minima when considered as single integrands. See for instance \cite[Assumptions (1.1)]{KM1} or just consider \rif{assF} with $p=q$. We remark that both the functional in \rif{model-2} and the one in \rif{model-3} cannot be covered by the present literature on functionals with $(p,q)$-growth as in \rif{pqgen}. Note indeed that in both cases the functional considered is non-differentiable and therefore its treatment cannot pass through the analysis of the related Euler-Lagrange equation. 

Sometimes we shall replace \rif{convex-omega} by the weaker
\eqn{convex-omega-2}
$$
\omega\colon [0, \infty)\to [0, \infty) \ \mbox{is concave and such that $\omega(0)=0$ and $\omega(\cdot)\leq 1$}\;,
$$
thereby considering functionals as in \rif{model-2} with $b(\cdot)$ being just continuous rather then H\"older continuous. 

The first two main results of this paper, stated in the next Theorem \ref{main1} and \ref{main2}, draw a complete parallel with the theory of functionals with standard polynomial growth \cite{GGInv, Manp, Manth} i.e., when $a(\cdot)\equiv 0$; see also \cite{KMarma, KMjems} for more recent a priori estimates. 
\begin{theorem}[Maximal regularity]\label{main1} Let $ u \in W^{1,p}(\Omega)$ be a local minimiser of the functional $\mathcal F$ defined in \trif{funF}, under the assumptions \trif{growtH} and \trif{assF}-\trif{convex-omega}. Moreover, assume that either \trif{bound1} or \trif{bound2} are satisfied. Then there exists $\beta_0 \in (0,1)$, depending only on $n,p,q,\ratio, \alpha$ and $\beta$, such that $Du \in C^{0,\beta_0}_{\rm{loc}}(\Omega; \er^n)$.
\end{theorem}
The main contribution in this paper is essentially threefold. First, we extend the results of \cite{CM1, CM2} to the largest possible family of functionals exhibiting a double phase behaviour of the type in \rif{model}, that is those functionals that can be controlled by the one in \rif{model} in the sense of \rif{growtH}; we further observe that we provide a unified and more transparent approach to the two different cases \rif{bound1}-\rif{bound2}. Second, the methods we develop here can provide a guideline to face the more general problem of regularity of minima under general Orlicz-type conditions. This means dealing with general non-autonomous functionals of the type in \rif{funF} where 
\eqn{condYoung}
$$F(x, v, z)\approx \Phi(x, |z|)$$ and $\Phi(\cdot)$ is a generalized Young function in the sense specified for instance in \cite{CH, Hasto}. An example is in fact given by the function $H(x, z)= K(x, |z|)\equiv \Phi(x, |z|)$ defined in \rif{HH}. The study of such problems has gained large attention over the last years, in particular with respect to the sharp interplay between the regularity of the function $x \mapsto \Phi(x,\cdot)$ and the growth conditions of $t \mapsto \Phi(\cdot,t)$. An example of this is indeed already given in \rif{bound1}-\rif{bound2}, while further results in this direction can be for instance found in \cite{AM1, AM2, Baroni, BL, Breit, BF, BO1, BO2, BOR, CMM, OK1, OK2, PS, TU}; see also the work of Ragusa \& Tachikawa for partial regularity \cite{RT1, RT2, TU}. In particular, we point out the recent paper \cite{HHT}, where a general setting for conditions \rif{condYoung} has been considered, with some unifying approaches and assumptions fitting several different contexts, as for instance those described in \cite{BCM2, ELM, Z1, Z2, Z3}. The issue is intriguing as assumptions that are relevant for regularity largely coincide with those that are necessary to get good functional theoretic properties of the spaces in questions; see for instance \cite{BCM2, CM1, CH} for this interplay. In this paper we use an intrinsic approach, useful to deal with cases like \rif{condYoung}. An instance of this is the Morrey type decay estimate \rif{morrey} below, formulated in terms of the natural quantity $H(\cdot, Du)$ and resembling the classical one valid for $p$-harmonic functions. It requires assumptions that are weaker than those considered in Theorem \ref{main1}. 
\begin{theorem}[Intrinsic Morrey decay]\label{main2} Let $ u \in W^{1,p}(\Omega)$ be a local minimiser of the functional $\mathcal F$ defined in \trif{funF}, under the assumptions \trif{growtH}, \trif{assF} and \trif{convex-omega-2}. Moreover, assume that either \trif{bound1} or \trif{bound2} are satisfied. Then, $u \in C^{0, \theta}_{\loc}(\Omega)$ for every $\theta<1$. Finally, for every $\sigma \in (0,n)$, there exists a positive constant $c\equiv c(\datao, \sigma)$, such that the decay estimate
\eqn{morrey}
$$
\int_{B_{\varrho}} H(x,Du)\, dx \leq c\left(\frac{ \varrho}{R}\right)^{n-\sigma}\int_{B_{R}} H(x,Du) \, dx
$$
holds whenever $B_\varrho \subset B_{R}\Subset \Omega_0$ are concentric balls with $R\leq 1$. \end{theorem}
The meaning of $\datao$ is clarified in \rif{idatio} below, see Section \ref{notazioni} for more notation. Theorem 
\ref{main2} has been obtained for the model case \rif{model} in \cite{CM1} under assumptions \rif{bound1} (but only when $q/p<1+\alpha/n$),  
but is new already for $\mathcal P$ when conditions in \rif{bound2} are considered. Another advantage of the methods presented here is that they allow to avoid the use of fractional estimates previously employed in \cite{CM1, CM2, ELM}. The use of such methods has a few drawbacks. As an example, it does not allow to extend the interior regularity results up to the boundary under the general assumptions one may wish to consider thinking to the classical case with $p$-growth. See for instance the recent interesting paper \cite{BO1}, where, due to the use of fractional estimates, additional regularity assumptions on the boundary have to be imposed in order to extend the interior results in \cite{CM3}.

\begin{remark}\label{borderline-re} The approach proposed here also allows to deal with the delicate borderline case in \rif{bound1}, i.e., 
\eqn{casoborder}
$$\frac qp=1+\frac{\alpha}{n}\;.$$
This has been left open in several papers on non-autonomous functionals \cite{BF, BO1, CM1, EMM, OK2}. Borderline cases are hard to catch. 
For instance, in the case of general functionals of the type 
$
w   \mapsto \int_{\Omega} F(Dw) \, dx$ 
with $(p,q)$-conditions of the type in \rif{pqgen} and suitable convexity assumptions, the  bound available in the literature is
$q/p< 1 +2/n$.  
It is not clear how to deal with the related limiting case $q/p= 1 +2/n$.
\end{remark}
\begin{remark}\label{parabolic-re} Another feature of the methods introduced here is that they open the way to treat parabolic equations of the type
$$
 u_t- \divo \, \left(|Du|^{p-2}Du+a(x,t)|Du|^{q-2}Du\right)=0\;.
$$
Such equations poses non-trivial additional difficulties as they generate new {\em double phase intrinsic geometries} and new methods must be developed; see \cite{BCMpar}. 
\end{remark}
The third goal of this paper is finally to disclose a new interpolative type phenomenon. In fact, we conjecture this should be a general principle when considering functionals with $(p,q)$-growth as in \rif{pqgen} and related non-uniformly elliptic problems. For this we shall consider simpler functionals of the type
\eqn{funF2}
$$
W^{1,1}(\Omega) \ni w   \mapsto \mathcal F(w, \Omega):=\int_{\Omega} F(x,Dw) \, dx \;,$$ 
but the results are completely new already in the model case \rif{model}. 
The idea is that  assuming more regularity of $u$ allows to further relax the bound linking $p$ and $q$. This is already visible in \rif{bound2} vs \rif{bound1}. The right scale to further quantify this phenomenon is the one of H\"older continuity, as shown in the next
\begin{theorem}[New interpolative bound]\label{main3} Let $ u \in W^{1,p}(\Omega)$ be a local minimiser of the functional $\mathcal F$ defined in \trif{funF2}, under the assumptions \trif{growtH}, \trif{assF}-\trif{convex-omega} and with $a(\cdot) \in C^{0, \alpha}(\Omega)$. Moreover assume that 
\eqn{bound3}
$$
u \in C^{0,\gamma}(\Omega) \qquad \mbox{and}\qquad q< p+\frac{\alpha}{1-\gamma}\,, \quad \gamma \in (0,1)\;.
$$
Then there exists $\beta_1 \in (0,1)$, depending only on $n,p,q,\ratio, \alpha$ and $\beta$, such that $Du \in C^{0,\beta_1}_{\rm{loc}}(\Omega; \er^n)$. Moreover, weakening assumption \trif{convex-omega} by \trif{convex-omega-2}, leads to the conclusions of Theorem \ref{main2}. 
\end{theorem}
The bound in \rif{bound3} exhibits the correct asymptotic. It formally reduces to \rif{bound2} when $\gamma\to 0$, while says that no bound on $q-p$ is needed when $\gamma\to 1$. Indeed, in this case we approach the Lipschitz continuity of $u$, which in fact makes the functional uniformly elliptic at infinity. As expected, the bound in \rif{bound3} improves the one in \rif{bound1} only when $p <  n/(1-\gamma)$, that is, only when the exponent $\gamma$ is better than the one naturally given by Sobolev-Morrey embedding theorem, since $
p > n/(1-\gamma)$ implies that  $u \in C^{0, \gamma}_{\loc}(\Omega)$. Only in this case $u \in C^{0, \gamma}(\Omega)$ becomes a real assumption and gives therefore an improvement.

Finally, the assumptions of Theorem \ref{main3} also allow to give new instances of absence of Lavrentiev phenomenon \cite{Lav, Z1, Z2, Z3}. According to the classical definition specialized to the present context, the Lavrentiev phenomenon for the functional $\mathcal F$ in \trif{funF} under assumptions \rif{growtH} occurs when 
\eqn{pheno}
$$
\inf_{w \in  u_0+W^{1,p}_0(B)}\, F(w, B ) <  \inf_{w \in  u_0+W^{1,p}_0(B) \cap W^{1,q}_{\loc}(B)}\, F(w, B)\;,
$$
where $B \Subset \Omega$ is a ball and $u_0 \in W^{1,\infty}(B)$. Considering the model functional $\mathcal P$ in \rif{model},  as shown in \cite{ELM}, it can be proved that \rif{pheno} actually occurs when \rif{bound1}-\rif{bound2} fail, 
and this eventually allows to give examples of irregular minima; see also \cite{FMM}. In \cite{ELM}  
this approach is reversed and the absence of Lavrentiev phenomenon - i.e., the non-occurrence of \rif{pheno} -
is used to prove regularity of minima. In particular, in \cite{ELM} the absence of Lavrentiev phenomenon for general functionals as in \rif{funF} is proved under assumptions \rif{bound1} and \rif{growtH}. 
When instead looking at  \rif{bound2}, the absence of Lavrentiev phenomenon for the model case $\mathcal P$ has been obtained in \cite{CM2} as a consequence of the available higher regularity of minima (not available for functionals as in \rif{funF} under conditions \rif{growtH}). Next theorem features a more general result, under less general assumptions already in the case \rif{bound2} is considered. Indeed, it only assumes \rif{growtH}. 
\begin{theorem}[Absence of Lavrentiev phenomenon]\label{main4} Let $ u \in W^{1,p}(\Omega)$ be a local minimiser of 
the functional $\mathcal F$ defined in \trif{funF}, under assumptions \trif{growtH} and with $a(\cdot) \in C^{0, \alpha}(\Omega)$. Assume that either \trif{bound2} or \eqn{bound33}
$$
u \in C^{0,\gamma}(\Omega) \qquad \mbox{and}\qquad q\leq  p+\frac{\alpha}{1-\gamma}\,, \qquad \gamma \in (0,1)
$$
holds. Then, for every ball
 $B\Subset  \Omega$, there exists a sequence $\{u_k\} $ of $W^{1,\infty}(B)$-regular functions such that $u_k \to u$ strongly in 
 $W^{1,p}(B)$ and such that 
 \eqn{convergenza1}
 $$
 \lim_{k} \, \mathcal F (u_k, B) =\mathcal F (u, B)\;.
 $$
\end{theorem}

We finally conclude giving an outline of the proofs, starting with the one of Theorem \ref{main1}. As already mentioned, a main new fact, allowing to catch the borderline case \rif{casoborder}, is that we are avoiding the use of fractional estimates and actually of any differentiation of the Euler-Lagrange equation of the functional \rif{funF} (fractional or standard). These are 
very common and heavy tools in the setting of functionals with non-standard growth conditions (see for instance \cite{Breit, BF, BO1, CM1, CM2, EMM, ELM}), that in fact do not allow to reach optimal assumptions in several situations. We instead employ a suitable blow-up argument relying on a quantitative version of a certain nonlinear harmonic type approximation lemma (Lemma \ref{p-harm} below) that, in the original linear version, goes back to the classical work of De Giorgi \cite{DG}; see \cite{DM} for an overview. This allows a more efficient separation of phases, following the terminology introduced in \cite{CM1, CM2}. Specifically, on smaller and smaller scales/balls $B_{R}$, we shall distinguish between the $p$-phase, where an inequality of the type $a(x)\lesssim [a]_{0, \alpha}R^\alpha$ holds, and the $(p,q)$-phase, that is when  
$ [a]_{0, \alpha}R^\alpha \ll a(x)$ instead occurs; see Section \ref{lefasi}. In the $p$-phase 
we blow-up the original minimizer $u$ from Theorem \ref{main1} in $B_R$, and see that, in a sense, it behaves as a minimizer of a functional with standard $p$-growth. From this we infer certain regularity estimates for $u$ in $B_R$. This is done in 
Section \ref{fasep}. In the $(p,q)$-phase, instead, we use a direct freezing argument and see that $u$ behaves in $B_R$ as a minimizer of an anisotropic functional whose growth is controlled by $w \mapsto  \int_{B_{R}} [|Dw|^p+a_0|Dw|^q]\, dx$, for some constant $a_0>0$. From this we again infer some regularity estimates in $B_R$; see Section \ref{fasepq}. It is important to see that, in order to accelerate the blow-up, 
and especially to catch the borderline case \rif{casoborder}, we have to use a few preliminary higher integrability and H\"older 
continuity results (see Theorems \ref{HO}-\ref{HI} below). Some nonlinear Calder\'on-Zygmund estimates in a non-standard setting are also employesd (see Theorem \ref{cz0} below). We then combine the treatment of the two phases via an exit-time argument described in Section \ref{morreysec}. This leads to Theorem \ref{main2}. We finally look 
at the gradient regularity. To achieve a good control of the constants we use a different separation of phases. We indeed consider a $p$-phase this time defined by $a(x)\lesssim [a]_{0, \alpha}R^{\alpha-s}$ and a $(p,q)$-phase which is still of the type 
$ [a]_{0, \alpha}R^{\alpha} \ll a(x)$. A suitable choice of the number $s\in [0,\alpha)$ and of the constants involved will eventually determine the H\"older continuity exponent $\beta_0$ of $Du$ appearing in Theorem \ref{main1}. The proof of Theorem \ref{main3} takes conceptually more effort and requires a double application of the blow-up lemma. The first time, this will be done in order to by-pass the fact that the assumed H\"older continuity in \rif{main3} is not quantitatively preserved under blow-up. This does not provide us with the higher integrability result (Theorem \ref{HI} below) allowing to apply Lemma \ref{p-harm} as for Theorem \ref{main1}. For this we proceed in two steps. First we apply a more traditional, non-quantitative version of the harmonic approximation to get higher H\"older continuity of $u$. Once this is done, we recover the missing H\"older continuity in the blow-up procedure. We can then proceed as for Theorem \ref{bound1}, with the quantitative version of the harmonic approximation.
\section {Notation and preliminaries}\label{notazioni}
In this paper, following a usual custom, we denote by $c$ a general constant larger than one. Different occurences from line to line will be still denoted by $c$, while special occurrences will be denoted by $c_1, c_2,  \tilde c$ or the like. Relevant
dependencies on parameters will be emphasised using parentheses, i.e., ~$c_{1}\equiv c_1(n,p,\ratio)$ means that $c_1$ depends on $n,p,\ratio$. We denote by $ B_r(x_0):=\{x \in \er^n \, : \,  |x-x_0|< r\}$ the open ball with center $x_0$ and radius $r>0$; when not important, or clear from the context, we shall omit denoting the center as follows: $B_r \equiv B_r(x_0)$. Very often, when not otherwise stated, different balls in the same context will share the same center. We shall also denote $B_1 = B_1(0)$
if not differently specified. Finally, with $B$ being a given ball with radius $r$ and $\gamma$ being a positive number, we denote by $\gamma B$ the concentric ball with radius $\gamma r$. With $\mathcal B \subset \er^{n}$ being a measurable subset with finite and positive measure $|\mathcal B|>0$, and with $g \colon \mathcal B \to \er^{k}$, $k\geq 1$, being a measurable map, we shall denote by  $$
   (g)_{\mathcal B} \equiv \mean{\mathcal B}  g(x) \, dx  := \frac{1}{|\mathcal B|}\int_{\mathcal B}  g(x) \, dx
$$
its integral average. With $f\colon \Omega \to \er$ and $\mathcal B \subset \Omega$, with $\gamma \in (0,1)$ being a given number, we shall denote
$$
[f]_{0,\gamma; \mathcal B} := \sup_{x,y \in \mathcal B, x \not= y} \, 
\frac{|f(x)-f(y)|}{|x-y|^\gamma}
\,,\qquad [f]_{0,\gamma} \equiv [f]_{0,\gamma; \Omega} \,.$$
\begin{remark}\label{local-re}We note that, since all the results from Theorems \ref{main1}-\ref{main4} are local in nature, when considering assumptions \rif{bound2} and \rif{bound3}, we can also assume that $u \in L^{\infty}_{\loc}(\Omega)$ and $u \in C^{0, \gamma}_{\loc}(\Omega)$. This can be done up to passing to open subsets $\Omega_0\Subset \Omega$ and restating all the results with an additional dependence of the constants on $\dista$. In the same way, with $u$ being the minimizer of Theorems \ref{main1}-\ref{main3}, we can assume that $H(\cdot, Du) \in L^1_{\loc}(\Omega)$. We prefer to start with the global formulations for the ease of exposition. 
\end{remark}
In order to shorten the notation, we shall express the dependence of the constants on the various basic parameters using the symbol $\data$. This is defined as a set of objects that will vary according to the assumptions considered as follows
\eqn{idati}
$$
\data \equiv 
\left\{ 
\begin{array}{ccc}
n,p,q,\ratio, \alpha, \omega(\cdot),  [a]_{0, \alpha}, \|H(\cdot,Du)\|_{L^1(\Omega)} &\mbox{when \rif{bound1} holds}\\[7 pt]
n,p,q,\ratio, \alpha, \omega(\cdot), [a]_{0, \alpha}, \|u\|_{L^\infty(\Omega)}&\mbox{when  \rif{bound2} holds}\\[7 pt]
n,p,q,\ratio, \alpha, \gamma, \omega(\cdot),  [a]_{0, \alpha}, [u]_{0, \gamma}&\mbox{when \rif{bound3} holds}\;.
\end{array}
\right.
$$
We also need a local version of $\data$. For this, with $\Omega_0\Subset \Omega$ being a fixed open subset, we denote by $\datao$ the above set of parameters in addition to $\dista$:
\eqn{idatio}
$$
\datao \equiv \data, \dista\;.
$$ 
\begin{remark} We shall sometimes consider functionals of the type in \rif{funF} under the only assumptions \rif{growtH}. In this case $\data$ and $\datao$ omit the specifications of $\omega(\cdot)$, since \rif{assF} are not considered. 
\end{remark}
In the following we shall often deal with the vector field
\eqn{Vpq}
$$
V_{t}(z):= |z|^{(t-2)/2}z \;,\qquad t\in \{p,q\}\;,
$$
so that 
$
V_p(z):= |z|^{(p-2)/2}z$ and $V_q(z):= |z|^{(q-2)/2}z$. 
This vector field is of common use to formulate the monotonicity properties of operators of $p$-Laplacean type and related integral functionals (see for instance \cite{KM1, KM2}). In this respect we record the following property
\eqn{monotonicity}
$$
\left|V_{t}(z_1)-V_{t}(z_2)\right|^2 \leq c \,\big\langle |z_1|^{t-2}z_1-|z_2|^{t-2}z_2, z_1-z_2\big\rangle\,,
$$
We shall also use the following equivalence:
\eqn{diffV}
$$
|V_{t}(z_1)-V_{t}(z_2)| \approx (|z_1|+|z_2|)^{(t-2)/2}|z_1-z_2|\;,
$$
that holds with involved constants depending only on $n,t$ (see for instance \cite{KM1, KM2}).

Finally, we summarize some basic terminology about so called generalized Orlicz-Sobolev spaces. Roughly speaking, these are Sobolev spaces defined by the fact that the distributional derivatives lie in a suitable Orlicz space (actually, a Orlicz-Musielak space according to the terminology in use \cite{diening}) rather than a Lebesgue spaces, as usual. Classical Sobolev spaces are then a particular case. Such spaces, and related variational problems, are for instance discussed in \cite{CH, diening, Hasto}, to which we refer for more details. Here we shall simply consider spaces related to the Young type function defined in \rif{HH}. 
We then define
\eqn{defi-sob}
$$
W^{1, H}(\Omega):= \left\{ u \in W^{1,1}(\Omega)\, \colon \, H(\cdot, Du)\in L^1(\Omega) \right\} \;,
$$
with the local variant being defined in the obvious way and $W^{1, H}_0(\Omega)=W^{1, H}(\Omega)\cap W^{1,p}_0(\Omega)$. The one in \rif{defi-sob} is in a sense the natural energy space associated to variational problems under growth conditions \rif{growtH}; see for instance \cite{HHT}. Finally, with $a_0 \geq 0$ being a non-negative number, we shall use the frozen Young function
\eqn{HH00}
$$
H_0(z) := |z|^p +a_0|z|^q=:K_0(|z|)\;, 
$$
with the notation fixed accordingly to the one used in \rif{HH}. 
Then $W^{1, H_0}(\Omega)$, and its related variants, can be defined exactly as in \rif{defi-sob} with the choice $a(x)\equiv a_0$. Note that $[0,\infty)\ni s \mapsto K_0(s)\in [0,\infty)$ is a strictly convex, monotone smooth bijection of $[0,\infty)$. Denoting  by $\widetilde{K_0}$ its Young's (or convex) conjugate, i.e., $\widetilde{K_0}(s):=\sup_{\tau>0} [s\tau-K_0(\tau)]$, then the following standard property
\eqn{prima-prop}
$$
\widetilde{K_0}\bigg(\frac{K_0(s)}{s}\bigg)\leq K_0(s)
$$
holds for every $s>0$ (see \cite{BL} for more details). We shall later use the following Young type inequality, valid for every $\varepsilon\in(0,1)$ and $s, t\geq 0$:
\begin{equation}\label{Young.eps}
st\leq \frac{K_0(s)}{\varepsilon^{q-1}}+\varepsilon\,\widetilde{K_0}(t)\;.
\end{equation}
We briefly report the proof. For $0 < \kappa \leq 1$ and for all $s\geq0$, since $q\geq p$, we have
\begin{eqnarray*}
\widetilde{K_0}(\kappa s) &= &\sup_{\tau>0}\left(\kappa s\tau-K_0(\tau)\right)\leq \sup_{\tau>0}\left(\kappa^{\frac{q}{q-1}}s\,\kappa^{-\frac1{q-1}}\tau-\kappa^{\frac{q}{q-1}}K_0(\kappa^{-\frac{1}{q-1}}\tau)\right)\\
&=& \sup_{\tilde\tau=\kappa^{-1/(q-1)}\tau>0}\big(\kappa^{\frac{q}{q-1}}s\,\tilde\tau-\kappa^{\frac{q}{q-1}}K_0(\tilde\tau)\big)=\kappa^{\frac{q}{q-1}}\widetilde{K_0}(s)\;.
\end{eqnarray*}
Thus, with $\varepsilon\in(0,1)$, we have the standard Young's inequality for conjugate functions yields 
$st\leq K_0(\eps^{(1-q)/q}s)+\widetilde{K_0}(\eps^{(q-1)/q}t)$ and \rif{Young.eps} follows using the property in the above display and again that $q\geq p$ while $\eps \in (0,1)$. 

\section{First regularity results}
We review a few basic regularity results available for minimizers of functionals with double phase. 
The main references here are \cite{BCM1, CM1, CM2, OK2} and the results are essentially proved there. We shall treat minimizers of functionals as in \rif{funF} under the only growth assumptions in \rif{growtH}. Here as in the rest of the paper, we shall denote by $\Omega_0$ an open subset such that $\Omega_0 \Subset \Omega$. We start by a Caccioppoli type inequality from \cite{CM1}. 
\begin{prop}\label{cacc-gen} Let $ u \in W^{1,p}(\Omega)$ be a local minimiser of the functional $\mathcal F$ defined in \trif{funF}, under the assumptions \trif{growtH}, with $a(\cdot)\in L^{\infty}_{\loc}(\Omega)$ and $u \in L^q_{\loc}(\Omega)$. Then there exists a constant depending only on $n,p, q,\ratio$ such that
\eqn{cacc0}
$$
\int_{B_{t}} H(x,Du)\, dx \leq c\int_{B_{s}}K\left(x,\left|\frac{u-u_0}{s-t}\right|\right) \, dx\;,
$$ 
holds whenever $B_t \Subset B_s \Subset \Omega$ are concentric balls and $u_0\in \er$, where $K(\cdot)$ has been introduced in \trif{HH}. 
In the same way, the following related inequality on level sets 
\eqn{cacc3---}
$$ \int_{B_{t}} H(x,D(u-k)_{\pm})\, dx \leq  
c\int_{B_{s}} K\left(x,\left|\frac{(u-k)_{\pm}}{s-t}\right|\right)\, dx  
$$
holds for every $k \in \er$, $t < s$, where $c\equiv c (n,p,q, \ratio)$ and  
\eqn{letroncate}
$$(u-k)_+:= \max\{u-k,0\}\qquad \mbox{and}\qquad (u-k)_-:= \max\{k-u,0\}\;.$$
\end{prop}
Notice that the condition $u \in L^q_{\loc}(\Omega)$ is always satisfied when one of the assumptions \rif{bound1}, \rif{bound2} and \rif{bound3} is in force. 

Next, a higher integrability result.
\begin{theorem}\label{HI} Let $ u \in W^{1,p}(\Omega)$ be a local minimiser of the functional $\mathcal F$ defined in \trif{funF}, under the assumptions \trif{growtH}. Moreover, assume that one of the conditions in \trif{bound1}, \trif{bound2} and \trif{bound3} is satisfied. Then there exists a higher integrability exponent 
$\delta\equiv \delta (\data)\in (0,1)$ such that $H(\cdot, Du)\in L^{1+\delta}_{\loc}(\Omega)$. Furthermore, the reverse type H\"older inequality 
\eqn{riversa}
$$
\left(\mean{B_{R/2}} [H(x,Du)]^{1+\delta}\, dx\right)^{1/(1+\delta)} \leq c 
\mean{B_{R}} H(x,Du)\, dx
$$
holds for a constant $c\equiv c(\data)$, whenever $B_{R}\Subset \Omega$ is a ball with $R\leq 1$. 
\end{theorem}

\begin{proof} This has already been obtained in \cite{CM1, CM2, OK2} as far as \rif{bound1}-\rif{bound2} are used. Here we cover the missing case of \rif{bound3}, essentially recalling the arguments for \cite[Theorem 1.2]{CM2}, to which we refer for more details. Consider a ball $B_R\Subset \Omega$ with $R\leq 1$ and consider the quantity $\ai(B_{R}):= \min_{\overline {B_R}}\, a(x)$. If 
$\ai(B_{R})> 4[a]_{\alpha} R^\alpha$ then, exactly as in \cite[Theorem 1.2]{CM2}, we find that the reverse type inequality
\eqn{inversa1}
$$
\mean{B_{R/2}}H(x, Du)\, dx \leq c \left(\mean{B_{R}}[H(x, Du)]^{d}\, dx  \right)^{1/d}\;,
$$
holds for $d:=\max\{{n}/({n+p}), 1/p\}\in (0,1)$, where the constant $c$ depends only on $n,p,q, \nu, L$. We now consider the remaining case $\ai(B_{R})\leq 4[a]_{\alpha} R^\alpha$, where we estimate as follows:
\begin{eqnarray}
\notag a(x)\left|\frac{u-(u)_{B_{R}}}{R}\right|^{q-p} &\leq& 8 [a]_{0, \alpha}R^{p-q+\alpha}[\osc_{B_{R}}\, u]^{q-p}\\
\notag &\stackleq{bound3} &c[u]_{0, \gamma;B_{R}}R^{p-q+\alpha+ \gamma(q-p)}\stackleq{bound3}  c(\data)\;.
\end{eqnarray}
The last inequality allows to conclude
$$\mean{B_{R}} a(x)\left|\frac{u-(u)_{B_{R}}}{R}\right|^q\, dx \leq c \mean{B_{R}} \left|\frac{u-(u)_{B_{R}}}{R}\right|^p\, dx$$ and combining this with \rif{cacc0} (with $s=R$, $t=R/2$, $u_0=(u)_{B_{R}}$), we get, by Sobolev-Poincar\'e inequality
\begin{eqnarray*}
\notag &&\mean{B_{R/2}} H(x,Du)\, dx \leq  c \mean{B_{R}} \left|\frac{u-(u)_{B_{R}}}{R}\right|^p\, dx\\  &&\quad \leq 
c\left(\mean{B_{R}} |Du|^{pd}\, dx\right)^{1/d} \leq  c\left(\mean{B_{R}} [H(x, Du)]^{pd}\, dx\right)^{1/d} \;,
\end{eqnarray*}
where $c\equiv c (\data)$ and $d$ is the same appearing in \rif{inversa1}. We conclude that $H(\cdot, Du)$ satisfies a reverse type H\"older inequality, that is \rif{inversa1} holds for every ball $B_{R} \Subset \Omega$ such that $R \leq 1$. At this point \rif{riversa} follows using a variant of Gehring's lemma on reverse H\"older inequalities \cite[Theorem 6.6]{G}. 
\end{proof}
Finally, we collect some H\"older continuity assertions from \cite{BCM1, CM1, OK2}. Special emphasis is put on the precise dependence of the various constants. 
\begin{theorem}\label{HO} Let $ u \in W^{1,p}(\Omega)$ be a local minimiser of the functional $\mathcal F$ defined in \trif{funF}, under the assumptions \trif{bound1} and \trif{growtH}. Then $u$ is locally H\"older continuous. Moreover, for every open subset $\Omega_0\Subset \Omega$, there exists a H\"older continuity exponent
\eqn{basicexp} 
$$\gamma\equiv \gamma\left(n,p,q,\ratio, \alpha, [a]_{0, \alpha}, \|u\|_{L^{\infty}(\Omega_0)}\right) \in (0,1)$$ such that
\eqn{basichol}
$$
\|u\|_{L^\infty(\Omega_0)}+ [u]_{0, \gamma;\Omega_0} \leq c (\datao)\;, 
$$
and the oscillation estimate 
\eqn{localstimahol}
$$\osc_{B_{\varrho}}\, u   \leq c \left(\frac{\varrho}{r}\right)^{\gamma}
\osc_{B_{r}}\,  u $$ holds for $c\equiv c \left(n,p,q,\ratio, \alpha, [a]_{0, \alpha}, \|u\|_{L^{\infty}(\Omega_0)}\right)$ 
and all concentric balls $B_\varrho \Subset B_{r} \Subset \Omega_0\Subset \Omega$ with $r \leq 1$. Finally, the dependence of the exponent in \trif{basicexp} can be reformulated as 
\eqn{basicexp2} 
$$\gamma\equiv \gamma\left(n,p,q,\ratio, \alpha, [a]_{0, \alpha}, \|H(\cdot, Du)\|_{L^{1}(\Omega)}, \dista\right) \in (0,1)\;.$$

The same conclusions hold when \trif{bound2} is assumed instead of \trif{bound1}; in this case the $L^\infty$-estimate in \trif{basichol} becomes immaterial as $u$ is assumed to be globally bounded. Moreover, the exponent $\gamma$ depends only on $\data$ and it is independent of the open subset $\Omega_0\Subset \Omega$ considered.
\end{theorem}
\begin{proof} The results in \cite[Section 10]{CM1} give that if $\|u\|_{L^{\infty}(\Omega)}$ is finite and $q \leq p+\alpha$, then $u\in C^{0, \gamma}_{\loc}$ and \rif{basichol}-\rif{localstimahol} hold. In this case $\gamma$ depends only on the parameters $n,p,q,\ratio, \alpha, [a]_{0, \alpha}, \|u\|_{L^\infty(\Omega)}$ and no dependence on $\dista$ occurs. Moreover, estimate \rif{localstimahol} follows from \cite[Theorem 3.8]{OK2} or directly from the Harnack inequality proved in \cite{BCM1}. 
In this respect, see how to derive estimates of the type in \rif{localstimahol} from Harnack inequalities in \cite[Section 7.9]{G}. 
This fully covers the case when \rif{bound2} comes into the play. We then consider \rif{bound1}, when $p\leq n$. By the results in \cite{BCM1, CM1} and in particular by \cite[Theorem 3.6]{OK2}, minimizers are locally bounded. Therefore, since $q\leq p+\alpha$ is implied by the bound in $q/p\leq 1+\alpha/n$ when $p\leq n$, then we can reduce to the case when $u$ is bounded and this gives \rif{basichol}-\rif{localstimahol} again; moreover, the dependence displayed in \trif{basicexp} follows in any case. To recast the alternative dependence in \rif{basicexp2}, 
we recall that the local bounds proved in \cite{CM1, OK2}, together with standard covering arguments, imply an estimate of the type
$$
\|u\|_{L^{\infty}(\Omega_0)} \leq c \left(\int_{\Omega} [H(x, Du) +1]\, dx\right)^{1/p}\;,
$$
where  $c $ depends on $n,p,q,\ratio, \alpha, [a]_{0, \alpha}, 
\|H(\cdot, Du)\|_{L^1(\Omega)}$ and $\dista$. Combining this fact with the dependence in \rif{basicexp} yields the required dependence in \rif{basicexp2}. Finally, when $p > n$, we can 
use the Harnack inequality proved in \cite{BCM1}, and this yields again \rif{localstimahol} for a constant $c$ and an exponent $\gamma$, being independent of $\dista$, but just depending on $n,p,q,\ratio, \alpha, [a]_{0, \alpha}, \|H(\cdot, Du)\|_{L^{1}(\Omega)}$.  
Notice that in the Harnack statements from \cite{BCM1} we have that the constants also depend on the diameter of $\Omega$. This dependence does not occur when one restricts the Harnack inequalities on balls $B_R$ with $R\leq 1$, which is the thing needed here in order to get \rif{localstimahol}. 
\end{proof}
\section{Initial setting for Theorems \ref{main1}-\ref{main2}}\label{setsec}
In this section we begin the proofs of Theorems \ref{main1}-\ref{main2} and fix the initial setting aimed at treating the cases \rif{bound1}-\rif{bound2} in a unified way. We first observe that
$
p+\alpha < p(1+\alpha/n)$ iff $p> n
$, 
so that \rif{bound1} provides a weaker assumption than \rif{bound3} when $p>n$. On the other hand, in this last case minimizers are automatically bounded. 
When $p=n$ we have that the relations between $p$ and $q$ in \rif{bound1}-\rif{bound2} do coincide and therefore \rif{bound1} becomes again an assumption weaker than \rif{bound2}, since it does not require that minimizers are a priori bounded. In this respect, notice that assuming $p=n \geq q-\alpha$ gives that minimizers are bounded by Theorem \ref{HO}. On the contrary, in the case $p<n$, assumptions \rif{bound1} imply \rif{bound2} and we can always reduce to consider \rif{bound2}. Indeed when $q/p\leq 1+\alpha/n$, minimizers are always locally bounded again by Theorem \ref{HO}, and $q/p\leq 1+\alpha/n$ implies that $q\leq p+\alpha$. Summarizing, we shall always make a distinction in the forthcoming proofs: we shall consider the case \rif{bound2} holds when $p<n$ holds, and the case when \rif{bound1} holds but only for $p\geq n$ (and vice-versa). Notice that this is perfectly consistent with the notation in \rif{idati}-\rif{idatio}. In fact, when $p<n$ and \rif{bound2} covers also the case \rif{bound1}, it happens that by Theorem \ref{HO} we can locally bound $\|u\|_{L^\infty}$ by $\|H(\cdot, Du)\|_{L^1}$, so that the final dependence is on this last quantity, exactly as prescribed in \rif{idati} when \rif{bound1} is considered. See the proof of Theorem \ref{HO} for more. 
\section{Quantitative approximation}\label{quant}
Harmonic type approximation lemmas give a way to perform blow-up procedures without passing to the limit. They are in use since the seminal work of De Giorgi \cite{DG}, and we refer to \cite{DM} for an overview of the subject. For recent non-standard version related to the setting of this paper we refer to \cite{DSV}, from which we will also borrow the direct approach used here (vs the indirect one presented in \cite{DG, DM}). In this section we give a version where the control of the constants can be made explicit and becomes of polynomial type. This is the effect of assuming initial higher integrability (see \rif{eqn:v-int002} below). In the following we shall deal with a general vector field $A_0\colon \er^n \to \er^n$, which is assumed to be $C^1(\er^n\setminus\{0\})$-regular and satisfying the following growth and ellipticity assumptions:
\eqn{ass0}
$$
\left\{
\begin{array}{c}
|A_0(z)||z|+|\partial A_0(z)||z|^2 \leq LH_0(z)\\ [8 pt]
\displaystyle \nu\frac{H_0(z)}{|z|^2}|\xi|^2\leq  \langle \partial A_0(z)\xi, \xi\rangle \;,
 \end{array}\right.
 $$
whenever $z \in \er^n\setminus \{0\}$, $\xi \in \er^n$, where $0 <\nu \leq 1 \leq L$ are fixed constants. The function $H_0(\cdot)$ is the one defined in \rif{HH00} for some $a_0\geq 0$. No condition is considered here on the two exponents $1<p < q$, which can be arbitrarily far from each other. We shall in the following use the strict monotonicity property:
\eqn{monotonicity-0}
$$|V_p(z_1)-V_p(z_2)|^2 + a_0 |V_q(z_1)-V_q(z_2)|^2 \leq c \langle A_0(z_1)-A_0(z_2), z_1-z_2\rangle 
$$
that holds for a constant $c\equiv c (n,p,q,\nu)\geq 1$, and for every $z_1, z_2 \in \er^n$; the maps $V_p(\cdot)$ and 
$V_q(\cdot)$ are defined in \trif{Vpq}. This is easily seen to be a consequence of \rif{ass0}$_2$ and \rif{diffV} (see \cite{CM1, KM1}). Related to the vector field $A_0(\cdot)$, we consider the following Dirichlet boundary value problem, which is naturally defined in the Sobolev space $W^{1, H_0}(B)$:
\eqn{Dir00}
$$
\left\{
\begin{array}{c}
-\divo\, A_0(Dh)=0 \quad \mbox{in} \  B\\[5 pt]
h\in  v +W^{1,H_0}_0(B) 
\end{array}
\right.
$$
where $B\subset \er^n$ is a given ball and $v \in W^{1,H_0}(B) $ is a given boundary datum. Solutions are meant in the usual distributional sense. Such a problem has already been considered in \cite[Section 5]{CM3}, to which we refer for existence. We then report the following Calder\'on-Zygmund type result:
\begin{theorem}[\cite{CM3}, Theorem 5.1]\label{cz0}
Let $h\in W^{1,1}(B)$ be a distributional solution to \trif{Dir00} such that $H_0(Dv),H_0(Dh) \in L^1(B)$, under the assumptions \trif{ass0}. Then
$$
H_0(Dv)\in L^t(B) \Longrightarrow H_0(Dh)\in L^t(B)\qquad \mbox{for every}\ 
t >  1\;.
$$
Moreover, for every $t > 1$, there exists a constant $c\equiv c(n,p,q,\ratio, t)$, which is in particular independent of $a_0$, such that the following inequality holds:
\eqn{bordo}
$$
\mean{B} [H_0(Dh)]^t \, dx  \leq 
c\mean{B} [H_0(Dv)]^t \, dx\;.
$$
\end{theorem}

Before stating the main result of this section we still need another fact, namely, a by now classical truncation lemma due to Acerbi \& Fusco \cite{AF}. The statement involves the Hardy-Littlewood maximal operator, defined as follows
\eqn{maximal-def}
$$
M(f)(x) := \sup_{B_r(x) \subset \er^n}\, \mean{B_r} |f(y)|\, dy\;,\quad x \in \er^n\;,
$$
whenever $f \in L^1_{\loc}(\er^n)$. We then have
\begin{theorem}[\cite{AF}]\label{thm:truncation}
Let $B \subseteq \er^n$ be a ball and $w \in W^{1,1}_0(B)$. Then for every $\lambda>0$ there exists $w_\lambda \in W^{1,\infty}_0(B)$ 
such that
\begin{equation}
\label{eqn:lip-w-lambda}
\| D w_{\lambda}\|_{L^\infty(B)} \leq c \lambda
\end{equation}
for some constant $c$ depending only on $n$. Moreover, it holds that 
\eqn{negligible}
$$\{ w_\lambda \neq w\} \subseteq B \cap \{ M(|\nabla w|) >\lambda\}\ \cup \ \mbox{negligible set}\;.
$$
\end{theorem}
Note that, in view of the definition used in \rif{maximal-def}, in the above theorem we can assume $w$ to be defined on the whole $\er^n$ by setting $w\equiv 0$ outside $B$. 

We are now ready to state the main result of this section. We remark in advance that a crucial point is that all the constants must be independent of the number $a_0$ appearing in \rif{HH00}.
\begin{lemma}[Harmonic type approximation]\label{p-harm}
Let $B\equiv B_r  \subseteq \er^n$ be a ball, let $\sigma \in (0,1)$ and let $v \in W^{1,p}(2B)$ be a function satisfying the following estimates:
\begin{equation}\label{eqn:v-int00}
 \mean{2B}H_0(Dv) \, dx \leq \tilde c_1
\end{equation}
and 
\begin{equation}\label{eqn:v-int002}
\mean{B}\big[H_0(Dv)\big]^{1+\delta_0} \, dx \leq \tilde c_2\;,
\end{equation}
where $\tilde c_1, \tilde c_2\geq 1$ and $\delta_0 >0$ are fixed constants. Moreover, assume that 
\begin{equation}
\label{eqn:almost-p-harm}
\left|\mean{B}\big \langle A_0(Dv), D\varphi \big \rangle\, dx\right|\leq \sigma \|D\varphi\|_{L^\infty(B)} 
\qquad\quad \text{holds for all $\varphi \in C_0^\infty(B)$}\;.
\end{equation}
Then there exists a function 
$h  \in v+ W^{1,H_0}_0(B)$ such that the following conditions are satisfied:
\begin{equation}\label{p-har}
\mean{B}\big\langle A_0(Dh), D\varphi \big\rangle\, dx=0\qquad\quad \text{holds for all $\varphi \in C_0^\infty(B)$},
\end{equation}
\begin{equation}\label{eqn:v-int1}
 \mean{B}\big[H_0(Dh)\big]^{1+\delta_0} \, dx \leq c(n,p,q,\ratio, \delta_0)\,\tilde c_2\, ,
\end{equation}
\begin{equation}\label{ts:sob-dist-h-u}
\mean{B}\left(|V_p(Dv)-V_p(Dh)|^2+a_0 |V_q(Dv)-V_q(Dh)|^2\right) \, dx \leq c\, \sigma^{s_1}\;,
\end{equation}
and
\eqn{diffa-vh}
$$
\mean{B} \left(\left|\frac{v-h}{r}\right|^q+a_0\left|\frac{v-h}{r}\right|^q\right) \, dx \leq c\, \sigma^{s_0}\;.
$$
In \trif{ts:sob-dist-h-u} and \trif{diffa-vh} the dependence of constants involved is as follows: 
$s_1:=s_1(p,q,\delta_0)>0$, $s_0:=s_0(n,p,q,\delta_0)>0$ and $c\equiv c(n,p,q,\ratio, \delta_0,\tilde c_1,\tilde c_2)\geq 1$. 
\end{lemma}

\begin{proof} By a standard approximation argument we notice that, if \rif{eqn:almost-p-harm} holds for every  $\varphi \in C_0^\infty(B)$, then it holds also for every  $\varphi \in W_0^{1,\infty}(B)$. We then divide the proof in three steps.

{\em Step 1: Truncation.} We define $h  \in W^{1,p}(B)$ as in \rif{Dir00}. 
The standard energy estimate in this case (see \cite[Theorem 3.1]{CM1}) and  \rif{eqn:v-int00} give
\begin{equation}\label{energy.estimate}
 \mean{B}H_0(Dh) \, dx \leq   c \mean{B}H_0(Dh) \, dx \leq c(n,p,q,\ratio)\, \tilde c_1
\end{equation}
for $c\equiv c (n,p,q, \ratio)$.
Similarly, but using first \rif{bordo} and then \rif{eqn:v-int002}, we get
\eqn{eqn:v-int1-2}
$$
\mean{B} [H_0(Dh)]^{1+\delta_0}\, dx\leq c\mean{B} [H_0(Dv)]^{1+\delta_0}\, dx\leq c(n,p,q,\ratio, \delta_0)\, \tilde c_2\;,
$$
and this proves \rif{eqn:v-int1}. We now set $w:= v-h\in W^{1,H_0}_0(B)$ and we let $\lambda\geq 1$ to be chosen later; we consider $w_{\lambda}\in W^{1,\infty}_0(B)$ given by Theorem~\ref{thm:truncation}, which satisfies
\rif{eqn:lip-w-lambda} and \rif{negligible}. 
Using such properties, Chebyshev's inequality and finally the maximal theorem, we deduce that
\begin{eqnarray}\label{eqn:est-superlev-w}
\frac{|\{w\neq w_{\lambda}\}|}{|B|}& \leq & \frac{|B \cap \{ M(|D w|) >\lambda\}|}{|B|} \notag\\
&\leq & \frac{1}{[K_0(\lambda)]^{1+\delta_0}}  \mean{B} \big[K_0\big(|M(Dw)|\big)\big]^{1+\delta_0}\, dx \notag\\
&\leq & \frac{c(n,p,q, \delta_0)}{[K_0(\lambda)]^{1+\delta_0}}  \mean{B} \big[H_0(Dw)\big]^{1+\delta_0}\, dx \notag\\
&\leq &\frac{c}{[K_0(\lambda)]^{1+\delta_0}}  \left[\mean{B} \big[H_0(Dv)\big]^{1+\delta_0}\, dx+\mean{B} \big[H_0(Dh)\big]^{1+\delta_0}\, dx\right]\notag\\
&\leq &\frac{c \, \tilde c_2}{[K_0(\lambda)]^{1+\delta_0}}  
\end{eqnarray}
with $c\equiv c(n,p,q,\ratio, \delta_0)$ and in the last line we have used \rif{eqn:v-int1} and \rif{eqn:v-int1-2}. Now we test the weak formulation of \rif{Dir00}$_1$ with $w_{\lambda}$ to get 
\begin{eqnarray}
\notag \mathcal T_1 &:= &\mean{B}\big\langle A_0(Dv)- A_0(Dh) , Dw_{\lambda} \big\rangle \chi_{\{w= w_{\lambda}\}}\, dx
\\
\notag &= & \mean{B}\big\langle A_0(Dv) , Dw_{\lambda} \big\rangle\, dx\\ 
&& \  -\mean{B}\big\langle A_0(Dv)- A_0(Dh)  , Dw_{\lambda} \big\rangle \chi_{\{w\neq w_{\lambda}\}}\, dx
=:\mathcal T_2+\mathcal T_3\;.\label{combiT}
\end{eqnarray}
Next, we estimate each term in the previous equality. By \rif{monotonicity-0} we have
$$
\mathcal T_1\geq\frac1c\mean{B}  \left(|V_p(Dv)-V_p(Dh)|^2 +a_0|V_q(Dv)-V_q(Dh)|^2 \right)\chi_{\{w= w_{\lambda}\}}\, dx 
$$
for $c\equiv c(n,p,q,\nu)$. 
Using \rif{eqn:almost-p-harm} and then \rif{eqn:lip-w-lambda} we have
$$|\mathcal T_2| \leq \sigma \| Dw_{\lambda}\|_{L^\infty(B)} \leq c \,\sigma\lambda\,.$$
Finally, for $\mathcal T_3$, we fix $\varepsilon\in(0,1)$ to be chosen and estimate
\begin{eqnarray*}
|\mathcal T_3| &\leq &
\mean{B}\left(|A_0(Dh)|+|A_0(Dv)|  \right)|Dw_{\lambda}| \chi_{\{w\neq w_{\lambda}\}} \, dx\\
&\stackleq{ass0} & L\| D w_{\lambda}\|_{L^\infty(B)}\mean{B}\left[\frac{H_0(Dv)}{|Dv|}+\frac{H_0(Dh)}{|Dh|}\right]  \chi_{\{w\neq w_{\lambda}\}}\, dx\\
&\stackrel{\rif{HH00}}{=} & L\| D w_{\lambda}\|_{L^\infty(B)}\mean{B}\left[\frac{K_0(|Dv|)}{|Dv|}+\frac{K_0(|Dh|)}{|Dh|}\right]  \chi_{\{w\neq w_{\lambda}\}}\, dx\\
&\stackleq{Young.eps} & \eps \mean{B}\left[\widetilde{K_0}\left(\frac{K_0(|Dv|)}{|Dv|}\right)+
\widetilde{K_0}\left(\frac{K_0(|Dh|)}{|Dh|}\right)\right] dx\\ && \quad +\frac{cK_0\left( \| D w_{\lambda}\|_{L^\infty(B)}\right)}{\eps^{q-1}}\,\frac{|\{w\neq w_{\lambda}\}|}{|B|}\\
&\stackrel{\rif{prima-prop},\rif{eqn:lip-w-lambda}}{\leq} & \varepsilon\mean{B}\left[H_0(Dv)+H_0(Dh)\right]\, dx+\frac{cK_0(\lambda)}{\eps^{q-1}}\,\frac{|\{w\neq w_{\lambda}\}|}{|B|}\\
&\stackrel{\rif{eqn:est-superlev-w}}{\leq} & \varepsilon\mean{B}\left[H_0(Dv)+H_0(Dh)\right]\, dx+\frac{c}{[K_0(\lambda)]^{\delta_0}\eps^{q-1}}\;.\end{eqnarray*}
By further using \rif{eqn:v-int00}, \rif{energy.estimate} and that $K_0(\lambda) \geq \lambda ^p$, we conclude with
$$
\mathcal T_3 \leq c\,\varepsilon+\frac{c}{\lambda^{p\delta_0}\eps^{q-1}}\;.
$$
Here it is $c\equiv c(n,p,q,\ratio,\tilde c_1,\tilde c_2, \delta_0)$. Merging the estimates for $\mathcal T_1,\mathcal T_2$ and $\mathcal T_3$ with \rif{combiT}, we deduce that
\begin{eqnarray}
\notag && \mean{B}  \left(|V_p(Dv)-V_p(Dh)|^2 +a_0|V_q(Dv)-V_q(Dh)|^2 \right)\chi_{\{w= w_{\lambda}\}}\, dx\\
&& \quad \leq c\,\left[\sigma\lambda+\varepsilon+\frac{1}{\lambda^{p\delta_0}\eps^{q-1}}\right]=:c\,S(\sigma,\lambda,\varepsilon) \label{CS}
\end{eqnarray}
for a constant $c$ depending on $n,p,q,\ratio,\tilde c_1,\tilde c_2, \delta_0$ and $\eps \in (0,1)$ has still to be chosen. Let us use  the short notation
\eqn{dedev}
$$
\mathcal V^2:=|V_p(Dv)-V_p(Dh)|^2 +a_0|V_q(Dv)-V_q(Dh)|^2
$$
and fix $\theta \in (0,1)$, again to be chosen. H\"older's inequality and \rif{CS} give
\begin{equation}\label{eqn:w-coinc}
\left(\mean{B} \mathcal V^{2\theta} \chi_{\{w= w_{\lambda}\}}\, dx \right)^{1/\theta}
\leq c \, S(\sigma,\lambda,\varepsilon)\;. 
\end{equation}
Again using H\"older inequality, we estimate as
\begin{eqnarray*}
&&\hspace{-5mm}\left(\mean{B}  \mathcal V^{2\theta}\chi_{\{w\neq w_{\lambda}\}}\, dx \right)^{1/\theta}
\leq c\left(\frac{|\{w\neq w_{\lambda}\}|}{|B|} \right)^{\frac{1-\theta}{\theta}}  \mean{B}  \mathcal V^{2}\, dx \\
&& \qquad \qquad \stackleq{eqn:est-superlev-w}  c\,[K_0(\lambda)]^{-\frac{(1-\theta)(1+\delta_0)}{\theta}}\mean{B} \left[H_0(Dv)+H_0(Dh)\right] \, dx\\ && \qquad  \quad\stackrel{\rif{eqn:v-int00},\rif{energy.estimate}}{\leq}c\,[K_0(\lambda)]^{-\frac{1-\theta}{\theta}} \leq  c\lambda^{-\frac{p(1-\theta)}{\theta}}\;,
\end{eqnarray*}
for a constant $c\equiv c (n,p,q,\ratio, \tilde c_1,\tilde c_2, \delta_0)$; notice we have used that $\lambda \geq 1$ and that $K_0(\lambda)\geq \lambda^p$. 
The inequalities in the last two displays give, after some manipulation
$$\left(\mean{B}  \mathcal V^{2\theta}\, dx \right)^{1/\theta}\leq c 
\,S(\sigma,\lambda,\varepsilon)+c\,{\lambda^{-\frac{p(1-\theta)}{\theta}}}\;,$$
and again $c\equiv c (n,p,q,\ratio, \tilde c_1,\tilde c_2, \delta_0)$. 
Recalling the definition of $S(\sigma,\lambda,\varepsilon)$ in \rif{CS}, and choosing 
$\lambda=\sigma^{-1/2}$ and $ \varepsilon=\sigma^{p\delta_0/[4(q-1)]}$, the previous inequality implies that
\eqn{eqn:sob-dist-u-h}
$$
\left(\mean{B}   \left(|V_p(Dv)-V_p(Dh)|^2 +a_0|V_q(Dv)-V_q(Dh)|^2\right)^{\theta}\, dx \right)^{1/\theta}\leq  c\,\sigma^{m}
$$
for $c:=c(n,p,q,\tilde c_1,\tilde c_2, \delta_0,\theta)$, where
\eqn{emme}
$$m\equiv m(\theta):= \min\left\{\frac{1}{2}, \frac{p\delta_0}{4}, \frac{p\delta_0}{4(q-1)}, \frac{p(1-\theta)}{2\theta} \right\}\;.$$
Recall that $\theta \in (0,1)$ is yet to be chosen. 

{\em Step 2: Proof of \trif{ts:sob-dist-h-u}.}
We show how to deduce \rif{ts:sob-dist-h-u} from \rif{eqn:sob-dist-u-h} by properly choosing $\theta \in (0,1)$. H\"older's inequality with conjugate exponents $2(1+\delta_0)/(1+2\delta_0)$ and $2(1+\delta_0)$ yields
\eqn{using-with}
$$
\mean{B} \big|\mathcal V\big|^2\, dx\leq \left(\mean{B}\big|\mathcal V\big|^{\frac{2(1+\delta_0)}{1+2\delta_0}}\, dx\right)^{\frac{1+2\delta_0}{2(1+\delta_0)}}\left(\mean{B}\big|\mathcal V\big|^{2(1+\delta_0)}\, dx\right)^{\frac1{2(1+\delta_0)}}\;.
$$
We now chose $\theta$ as 
\eqn{iltheta}
$$\theta\equiv \theta(\delta_0):=\frac{1+\delta_0}{1+2\delta_0}\in(0,1)\;.$$ This allows to estimate
\[
 \left(\mean{B}\big|\mathcal V\big|^{\frac{2(1+\delta_0)}{1+2\delta_0}}\, dx\right)^{\frac{1+2\delta_0}{2(1+\delta_0)}}\leq c\,\sigma^{m/2}.
\]
On the other hand, recalling \rif{dedev}, we simply note that
\begin{align*}
 \mean{B}\big|\mathcal V\big|^{2(1+\delta_0)}\, dx&=\mean{B}\Big(\big|V_p(Dh)-V_p(Dv)\big|^2 +a_0\big|V_q(Dh)-V_q(Dv)\big|^2\Big)^{1+\delta_0}\, dx\\
&\leq c\,\mean{B}\big[H_0(Dv)\big]^{1+\delta_0}\, dx+c\, \mean{B}\big[H_0(Dh)\big]^{1+\delta_0}\, dx\\
&\leq c(n,p,q,\ratio, \delta_0,\tilde c_2)\;,
\end{align*}
where we have used \rif{eqn:v-int1} and \rif{eqn:v-int1-2} in the last line. Combining the content of the last display with \rif{using-with} and recalling \rif{emme} yields \rif{ts:sob-dist-h-u} with 
$$s_1:= \frac{m(\theta(\delta_0))}{2}\equiv \frac 12\min\left\{\frac{1}{2}, \frac{p\delta_0}{4}, \frac{p\delta_0}{4(q-1)}, \frac{p\delta_0}{2(1+\delta_0)} \right\}\;,$$ that in fact exhibits the announced dependence on $p,q$ and $\delta_0$.

{\em Step 3: Proof of \trif{diffa-vh}.} We again use \trif{eqn:sob-dist-u-h} and show how this implies \rif{diffa-vh} for a suitable choice of $\theta \in (0,1)$. If $q \geq 2$, by \rif{diffV} we see that taking $\theta \equiv \theta_q= \max\{{n}/({n+q}), 1/q\}$
\begin{eqnarray}
\notag a_0\mean{B} \left|\frac{v-h}{r}\right|^q \, dx   &\leq &c\, a_0\left(\mean{B} |Dv-Dh|^{q\theta}\, dx \right)^{1/\theta}
\\ &\leq & c\, \left(\mean{B} \left(a_0|V_q(Dv)-V_q(Dh)|^{2}\right)^\theta \, dx\right)^{1/\theta}  \leq c \sigma^{m(\theta_q)}
\label{comequi2}
\end{eqnarray}
for a constant $c$ depending only on $n,p$ and $\theta$. When $q<2$, we instead argue as follows:
\begin{eqnarray}
\nonumber && \hspace{-7mm}a_0\mean{B} \left|\frac{v-h}{r}\right|^q \, dx\leq a_0\left(\mean{B} |Du-D  h|^{q\theta}\, dx \right)^{1/\theta}
\\ \notag&&\stackleq{diffV} ca_0\left(\mean{B} \left(|V_q(Du)-V_q(D h)|^q (|Du|+|D \tilde h|)^{q(2-q)/2}\right)^{\theta} \, dx \right)^{1/\theta}
\nonumber 
\\
&& \, \, \leq c\left(\mean{B} \left(a_0|V_q(Du)-V_q(D h)|^2\right)^{\theta}\, dx\right)^{q/(2\theta)}\notag \\ && \qquad \cdot\left(\mean{B} \left(a_0(|Du|+|D h|)^{q}\right)^{\theta} \, dx\right)^{(2-q)/(2\theta)}\leq c \sigma^{m(\theta_q)q/2}\label{comequi3}\;,\end{eqnarray}
where in the last line we have used also \rif{eqn:v-int00} and \rif{eqn:v-int1}. Summarizing, we have  
$$
a_0\mean{B} \left|\frac{v-h}{r}\right|^q \, dx \leq c\sigma^{m(\theta_q)\min\{1, q/2\}}\;.
$$
In a completely similar way we can estimate 
$$
 \mean{B} \left|\frac{v-h}{r}\right|^p \, dx \leq c\,\sigma^{m(\theta_p)\min\{1, p/2\}}\;,
$$
 for $\theta_p:= \max\{{n}/({n+p}), 1/p\}$, and the proof of \rif{diffa-vh} follows for a suitable exponent $s_0= m(\theta_p)\min\{1, p/2\}$, with the dependence described in the statement. 
  \end{proof}
Beside the above quantitative harmonic approximation, we report the following more standard version, which is close to the ones described in \cite{DSV, DM}. 
\begin{lemma}\label{p-harm2}
Let $B \equiv B_r  \subseteq \er^n$ be a ball, land let $v \in W^{1,p}(B)$ be a function satisfying 
$$ \mean{B}H_0(Dv) \, dx \leq \tilde c_1$$
for some $\tilde c_1\geq 1$. Fix $\tilde \eps \in (0,1]$. There exists $\sigma \equiv \sigma (n,p,q,\ratio, \tilde \eps)$, but otherwise independent of $a_0$, such that if \trif{eqn:almost-p-harm} is satisfied, then there exists 
$h  \in v+ W^{1,H_0}_0(B)$ such that \trif{p-har} holds together with the estimates
\begin{equation}\label{eqn:v-int1-dopo}
 \mean{B}H_0(Dh)\, dx \leq c(n,p,q,\ratio)\,\tilde c_1
\end{equation}
and
\eqn{diffa-vh-dopo}
$$
\mean{B} \left(\left|\frac{v-h}{r}\right|^p+a_0\left|\frac{v-h}{r}\right|^q\right) \, dx \leq \tilde \eps\;.
$$
\end{lemma}
\begin{proof} This is an approximation lemma of the type already developed in \cite{DSV, DM} and the proof can be obtained as in these papers, with some modifications from the proof of Lemma \ref{p-harm}. This leads to establish that, actually, for every choice of $\tilde \eps, \theta \in (0,1) $, there exists $\sigma \equiv \sigma (n,p,q,\ratio, \tilde \eps, \theta)$ such that if \trif{eqn:almost-p-harm} holds, then there exists 
$h  \in v+ W^{1,H_0}_0(B)$ as in the statement of Lemma \ref{p-harm2}, and such that
\begin{equation}\label{comequi1}
\mean{B}\left(|V_p(Dv)-V_p(Dh)|^2+a_0 |V_q(Dv)-V_q(Dh)|^2\right)^{\theta} \, dx \leq \sigma\;.
\end{equation}
Once this is available, then \rif{diffa-vh-dopo} can be deduced from \rif{comequi1} as done using \rif{eqn:sob-dist-u-h} in Step 3 from the proof of Lemma \ref{p-harm}.
\end{proof}

\section{Two phases}\label{lefasi} From now, in the rest of the proof of Theorems \ref{main1} and \ref{main2}, we consider a fixed open subset $\Omega_0\Subset \Omega$; accordingly, we consider $\datao$ as defined in \rif{idatio}. By the basic results of Theorem \ref{HI} and \ref{HO}, we know $H(\cdot,Du) \in L^{1+\delta}_{\loc}(\Omega)$ 
and $u \in C^{0,\gamma}_{\loc}(\Omega)$ hold for higher integrability and H\"older continuity exponents $\delta >0$ and $\gamma \in(0,1)$, depending on $\data$ and $\datao$, respectively (see in fact \rif{basicexp}-\rif{basicexp2}). Notice that in the case we are assuming \rif{bound2}, then $\gamma$ only depends on $\data$ as defined in \rif{idati}. Moreover, the following estimate holds:
\eqn{stimelocali}
$$
[u]_{0, \gamma;\Omega_0}+\|H(\cdot,Du)\|_{L^{1+\delta}(\Omega_0)} \leq c\left(\datao\right)
$$
as a consequence of \rif{riversa}, \rif{basichol} and of standard covering arguments. To proceed, in the following we shall consider a number $s$ such that
\eqn{esseemme}
$$
0\leq s < s_m :=
\left\{
\begin{array}{cc}
\displaystyle \alpha-\frac{(q-p)n}{p}+\frac{n\delta(q-p)}{p(1+\delta)}& \mbox{if \rif{bound1} is considered and $p>n$}  \\[13 pt]
p+\alpha -q + \gamma(q-p) & \mbox{if \rif{bound2} is considered\;.}
\end{array}
\right.
$$ 
Needless to say, \rif{bound1}-\rif{bound2} imply that $s_m>0$, and in the limiting equality cases too; moreover, it is $s<s_m\leq \alpha$. 

Next, given a ball $B\equiv B_{R}\Subset \Omega$, $R\leq 1$, and a number $M\geq 1$, and with
\eqn{applica0}
$$
x_B \in \bar B\,,  \quad  a_{{\rm i}}(B):=a(x_B) =\min_{\bar B} \, a(x) \;,
$$ we define the $p$-phase in $B_{R}$ as the occurrence of the following inequality:
\eqn{applica-p}
$$
a_{{\rm i}}(B_{R}) \leq M[a]_{0,\alpha}R^{\alpha-s}\;.$$ 
The $(p,q)$-phase is instead defined by the complementary condition, i.e.,  
\eqn{applica}
$$
a_{{\rm i}}(B_{R}) > M[a]_{0,\alpha}R^{\alpha-s}\;.$$ 
\begin{remark}\label{propaga} 
We record an immediate consequence of \rif{applica-p}, that is
\eqn{condiadopo}
$$
a_{{\rm i}}(B_{R}) \leq M[a]_{0,\alpha}R^{\alpha-s} \Longrightarrow \|a\|_{L^\infty(B_{R})} \leq 3M [a]_{0,\alpha}R^{\alpha-s}\;.
$$
Indeed, observe that 
$$
\|a\|_{L^\infty(B_{R})}\leq a_{{\rm i}}(B_R) + [a]_{0,\alpha}(2R)^\alpha \leq M [a]_{0,\alpha}R^{\alpha-s}+[a]_{0,\alpha}(2R)^{\alpha}\leq 3M [a]_{0,\alpha}R^{\alpha-s}.
$$
\end{remark}
\begin{remark}\label{propaga2}
We observe that if the condition \rif{applica} 
holds in $B_{R}\equiv B_R(x_0)$, then it also holds on all smaller concentric balls $B_{\varrho}(x_0)$, that is
\eqn{ognir}
$$
a_{{\rm i}}(B_{R}(x_0)) > M[a]_{0,\alpha}R^{\alpha-s} \Longrightarrow a_{{\rm i}}(B_\varrho(x_0)) 
> M[a]_{0,\alpha}\varrho^{\alpha-s}
\qquad \forall \ \varrho \leq  R\;.
$$ Indeed we have 
$a_{{\rm i}}(B_\varrho(x_0)) \geq a_{{\rm i}}(B_{R}(x_0))
>M[a]_{0,\alpha} R^{\alpha-s}\geq M[a]_{0,\alpha} \varrho^{\alpha-s}.$
\end{remark}
\begin{remark}\label{propaga4} We notice the property
\eqn{ult}
$$a_{{\rm i}}(B_{R}) > M[a]_{0,\alpha}R^{\alpha-s} \Longrightarrow \|a\|_{L^{\infty}(B_R)} \leq [2/M+1] a_{{\rm i}}(B_R)$$
Indeed, 
$
a(x)\leq a(x)-a_{{\rm i}}(B_R) + a_{{\rm i}}(B_R) \leq 2[a]_{0, \alpha}R^\alpha + a_{{\rm i}}(B_R)\leq [2/M+1] a_{{\rm i}}(B_R).
$
\end{remark}
We gather the following consequence of Proposition \ref{cacc-gen}:
\begin{lemma}\label{cacc-gen2} Let $ u \in W^{1,p}(\Omega)$ be a local minimiser of the functional $\mathcal F$ defined in \trif{funF}, under the assumptions \trif{growtH}. Moreover, assume that either \trif{bound1} or \trif{bound2} is satisfied. Consider a ball $B_{R} \Subset \Omega_0\Subset \Omega$, with $R \leq 1$, such that 
\trif{applica-p} holds for some $s$ as in \trif{esseemme} and some $M\geq 1$. Then 
\eqn{cacc1}
$$
\mean{B_{t}} H(x,Du)\, dx \leq c \left(\frac R{s-t}\right)^q\mean{B_{s}}\left|\frac{u-u_0}{R}\right|^p\, dx
$$ 
holds for a constant $c\equiv c (\datao, M)$, whenever $B_t \Subset B_s\Subset B_{R}$ 
are concentric balls and for every $u_0\in \er$ such that $\inf u \leq u_0 \leq \sup u$.
Similarly, the inequality
\eqn{cacc3--}
$$ \int_{B_{t}} H(x,D(u-k)_{\pm})\, dx \leq  
c \left(\frac R{s-t}\right)^q\int_{B_{s}} \left|\frac{(u-k)_\pm}{s-t}\right|^p\, dx  
$$
holds for every $k \in \er$ such that $\inf u \leq k \leq \sup u$, where $c\equiv c (n,p,q, \ratio, M)$. Here $(u-k)_+$ and $(u-k)_-$ are defined as in \trif{letroncate}. 
\end{lemma}
\begin{proof} We give the proof of \rif{cacc3--}, the one of \rif{cacc1} being completely similar. An easy manipulation of \rif{cacc3---}, which holds in every case/phase, gives 
\begin{eqnarray}
\notag &&\int_{B_{t}}H(x, D(u-k)_\pm)\, dx\\ && \qquad  \leq 
c\left(\frac R{s-t}\right)^q\int_{B_{s}}\left(1+a(x)\left|\frac{(u-k)_\pm}{R}\right|^{q-p}\right)\left|\frac{(u-k)_\pm}{R}\right|^{p}\, dx\;.\label{manimani}
\end{eqnarray}
To proceed, we distinguish between the two cases, \rif{bound1} and \rif{bound2}. For \rif{bound2}, recalling that $\inf u \leq k \leq \sup u$ and using Remark \ref{propaga}, we can estimate
\eqn{manimani2}
$$
a(x)\left|\frac{(u-k)_\pm}{R}\right|^{q-p} \leq cM [a]_{0, \alpha} [u]_{0, \gamma;B_{R}}^{q-p}R^{p+\alpha -q + \gamma(q-p)-s} \stackleq{stimelocali} c \;,
$$ 
where $c\equiv c(\datao,M)$ and we have used the relation in \rif{esseemme}. 
Combining the inequalities in the last two displays yields \rif{cacc3--}. The we consider \rif{bound1}. By the discussion in Section \ref{setsec} we restrict to the case $p\geq n$. Indeed, by Sobolev-Morrey theorem, using again Remark \ref{propaga}, we have that 
\begin{eqnarray*}
a(x)\left|\frac{(u-k)_\pm}{R}\right|^{q-p} &\leq& 3M [a]_{0, \alpha}R^{p-q+\alpha-s}[\osc_{B_r}\, u]^{q-p}\\
&\leq &cR^{p-q+\alpha-s}\|Du\|_{L^{p(1+\delta)}(B_{R})}^{q-p}R^{\left[1-\frac{n}{p(1+\delta)}\right](q-p)}\\
&\leq & c\|H(\cdot, Du)\|_{L^{1+\delta}(B_{R})}^{q/p-1}R^{\alpha-\frac{(q-p)n}{p}+ \frac{n\delta(q-p)}{p(1+\delta)}-s} \stackleq{stimelocali} c\;,
\end{eqnarray*}
for $c\equiv c(\datao,M)$. 
Combining this last inequality with \rif{manimani} yields \rif{cacc3--} once again and the proof is complete. 
\end{proof}
\begin{remark}\label{stupido00} Estimates \rif{cacc1} and \rif{cacc3--} continue to hold if, instead of assuming \rif{bound1} or \rif{bound2}, we consider \rif{bound3}. In this case, instead of \rif{esseemme}, we always take $s_m= p+\alpha -q+\gamma(q-p)$. The validity of \rif{cacc3--} (and similarly of \rif{cacc1}) then follows by first writing \rif{manimani} and then observing than \rif{manimani2} still holds under assumptions \rif{bound3}.
\end{remark}

\section{$p$-phase}\label{fasep} In the setting specified in the previous section, we shall here exploit the estimates implied by condition \rif{applica-p}, where $B_{R} \Subset \Omega_0\Subset \Omega$ with $R\leq 1$, and $s$ is a number satisfying \rif{esseemme}. We define, whenever $u \in W^{1,1}(\Omega)$ is a local minimizer of \rif{funF} and $w \in W^{1,1}(B_{R})$
\eqn{frolast}
$$
\mathcal G_{B}(w):=   \int_{B_{R}} F(x, (u)_{B_{R}}, Dw) \, dx\;,$$
and $\tilde v \in u + W^{1,p}_0(B)$ as the solution to the variational Dirichlet problem:
\eqn{var-dir1}
$$
\left\{
    \begin{array}{c}
   \displaystyle  \tilde v \mapsto \min_w
\, \mathcal G_{B}(w)    
\\[8pt]
        w \in u + W^{1,p}_0(B_{R})\;. 
\end{array}\right.
$$ 
We then have the following preliminary lemma, that essentially serves to get rid of the $u$-dependence in the original functional $\mathcal F$ in \rif{funF}. 
\begin{lemma}\label{pqcomp-p} Let $ u \in W^{1,p}(\Omega)$ be a local minimiser of the functional $\mathcal F$ defined in \trif{funF}, under assumptions \trif{growtH}, \trif{assF} and \trif{convex-omega-2}. Assume that either \trif{bound1} or \trif{bound2} is satisfied. Let $B_{R}\Subset \Omega_0\Subset \Omega$ be a ball with $R\leq 1$. 
Then there exists a positive constant $c\equiv c\left(\datao\right)$ such that then the following estimate holds:
\begin{eqnarray}
\nonumber &&\mean{B_{R}} \left( |V_p(Du)-V_p(D\tilde v)|^2+ a(x) |V_q(Du)-V_q(D\tilde v)|^2\right)\, dx\\
&&\quad  \leq c\omega(R^\gamma) \mean{B_{R}} H(x, Du) \, dx\;, 
\label{comp11}
\end{eqnarray}
where $\gamma\equiv \gamma (\datao)\in (0,1)$ is the H\"older continuity exponent fixed in \trif{stimelocali}. Moreover, the following inequalities hold:
\eqn{infinito}
$$
\|\tilde v\|_{L^{\infty}(B_{R})}\leq \|u\|_{L^{\infty}(B_{R})}\;,\qquad \mean{B_{R}}H(x, D\tilde v )\, dx\leq  \frac L\nu \mean{B_{R}}H(x, Du)\, dx
$$
and 
\eqn{stimaosc}
$$
\osc_{B_{R}}\, \tilde v \leq \osc_{B_{R}}\, u\;. 
$$
Finally, for every $\eps \in (0,1)$, the estimate
\eqn{comp11-ancora}
$$
\mean{B_{R}} \left|\frac{u-\tilde v}{R}\right|^p\,  dx\leq c\, 
\left[c_\eps\, \omega(R^\gamma) + \eps\right] \mean{B_{R}}H(x, Du)\, dx 
$$
holds for $c$ depending on $\datao$ and $c_\eps$ depending only on $\eps, p$ and $q$. 
\end{lemma}
\begin{proof} The Euler-Lagrange equation of the functional $\mathcal F_B$, that is 
\eqn{el1}
$$
\mean{B_{R}} \big\langle  \partial F(x, (u)_{B_{R}}, D\tilde v), D\varphi \big\rangle \, dx=0\;,$$
is now satisfied for any choice  $\varphi \in W^{1,1}_0(B_R)$ such that $H(x, D\varphi)\in L^1(B_R)$ (see also \cite[Remark 6.1]{CM3}). On the other hand, we notice that 
by minimality and growth conditions we also have that $H(x, D\tilde v)\in L^1(B_R)$:
\begin{eqnarray}
\notag \mean{B_{R}}H(x, D\tilde v )\, dx &\stackleq{growtH} & \frac1\nu \mean{B_{R}}F(x, (u)_{B_{R}}, D\tilde v)\, dx \\
&\leq & \frac1\nu  \mean{B_{R}}F(x, (u)_{B_{R}}, Du)\, dx \leq \frac L\nu \mean{B_{R}}H(x, Du)\, dx\;.\label{coerc}
\end{eqnarray}
This proves the second inequality in \trif{infinito}. 
Therefore we conclude with
\eqn{el1-p}
$$
\mean{B_{R}} \big\langle  \partial F(x, (u)_{B_{R}}, D \tilde v), Du-D\tilde v \big\rangle \, dx=0\;.$$
Moreover, let us recall two basic consequences of \rif{assF}$_2$. Using the content of \cite[(4.13)]{KM1}, and arguing in a totally similar way to get that lemma, we this time have that 
\begin{eqnarray}
&&|V_p(z_2)-V_p(z_1)|^2+ a(x) |V_q(z_2)-V_q(z_1)|^2\notag\\
&& \quad +c\big \langle  \partial F(x, (u)_{B_{R}}, z_1), z_2-z_1\big\rangle \leq c\left[ 
F(x, (u)_{B_{R}}, z_2)-F(x, (u)_{B_{R}}, z_1)\right]\label{el2}
\end{eqnarray}
holds for every choice of $z_1, z_2\in \er^n$ and $x \in \Omega$, for a constant $c$ depending only on $n,p,q,\nu$. 
By using \rif{el1}, \rif{el2} and the minimality of $u$ and $\tilde v$, we then find
\begin{eqnarray}
&&\notag \hspace{-5mm}  \mean{B_{R}} \left(|V_p(Du)-V_p(D\tilde v)|^2+ a(x) |V_q(Du)-V_q(D\tilde v)|^2\right) \, dx\\ 
&&\notag \,  \stackrel{\rif{el1-p}}{=}  \mean{B_{R}} \left(|V_p(Du)-V_p(D\tilde v)|^2+ a(x) |V_q(Du)-V_q(D\tilde v)|^2\right) \, dx \\
&&\notag \qquad +  c\mean{B_{R}} \big \langle  \partial F(x, (u)_{B_{R}}, D\tilde v), Du-D\tilde v\big \rangle \, dx \\
&& \notag \stackleq{el2} c\mean{B_{R}} \left[ 
F(x, (u)_{B_{R}}, Du)-F(x, (u)_{B_{R}}, D\tilde v)\right] \, dx \\
&& \notag \ \, =   c\mean{B_{R}} \left[ 
F(x, (u)_{B_{R}}, Du)-F(x, u, Du)\right] \, dx \\
&& \notag \qquad + c\mean{B_{R}} \left[ 
F(x, u, Du)-F(x, \tilde v, D\tilde v)\right] \, dx \\
&& \notag \qquad +  c\mean{B_{R}} \left[ 
F(x, \tilde v, D\tilde v)-F(x, (\tilde v)_{B_{R}}, D\tilde v)\right] \, dx \\
&& \qquad +  c\mean{B_{R}} \left[ 
F(x, (\tilde v)_{B_{R}}, D\tilde v)-F(x, (u)_{B_{R}}, D\tilde v)\right] \, dx =: c \sum_{i=1}^4 I_i\;.\label{connecting}
\end{eqnarray} 
We can now proceed estimating the terms $I_1, \dots, I_4$. We have
\begin{eqnarray}
\notag I_1 &\stackleq{assF} & c \mean{B_{R}}\omega\left(|u-(u)_{B_{R}}|\right)H(x, Du)\, dx \\
\notag &\stackleq{applica} & c\omega\left(2[u]_{0, \gamma;\Omega_0}R^\gamma\right) \mean{B_{R}}H(x, Du)\, dx\\
&\stackleq{basichol} &c(\datao)\,\omega (R^\gamma) \mean{B_{R}}H(x, Du)\, dx\;, \notag
\end{eqnarray}
where we have also used that $\omega(\cdot)$ is concave. 
On the other hand, the minimality of $u$ yields $I_2\leq 0$. As for $I_3$, we start proving 
\rif{stimaosc}. 
Indeed, observe that, with $k := \sup_{B_{R}} \, u$, the minimality of $\tilde  v$ yields
$
\mathcal G_{B}(\tilde v)  \leq \mathcal G_{B}(\min\{\tilde v,k\}) 
$
and this, together with \rif{growtH}, in turn gives 
$$
\int_{\{\tilde v \geq k\} \cap B} H(x, D\tilde v) \, dx =0\;.
$$
By coarea formula we then get $\tilde v \leq k$ a.e. Arguing similarly, but with the choice $k := \inf_{B_{R}} \, u$, and using that 
$
\mathcal G_B(\tilde v)  \leq \mathcal G_B(\max\{\tilde v,k\}), 
$
we get that $\tilde v \geq k$ a.e. and \rif{stimaosc} follows. This also proves the first inequality in \rif{infinito}. We therefore have 
\begin{eqnarray*}
\notag I_3 &\stackleq{assF} & c \mean{B_{R}}\omega\left(|\tilde v-(\tilde v)_{B_{R}}|\right)H(x, D\tilde v)\, dx \\
&\stackleq{stimaosc} & c\, \omega \left(\osc_{B_{R}} \, u\right)\mean{B_{R}}H(x, D\tilde v)\, dx \\
&\stackleq{coerc} & c\omega\left(2[u]_{0, \gamma;\Omega_0}R^\gamma\right) \mean{B_{R}}H(x, Du)\, dx\\
&\stackleq{basichol} &c\omega (R^\gamma) \mean{B_{R}}H(x, Du)\, dx\;, \notag
\end{eqnarray*}
where $c\equiv c (\datao)$. 
In a similar way, as $u\equiv \tilde v$ on $\partial B$, we first observe that 
$$ \left|(\tilde v)_{B_{R}} -(u)_{B_{R}}\right|^p\, dx \leq  c \mean{B_{R}}|\tilde v -u|^p\, dx \stackleq{stimaosc} c\left(\osc_{B_{R}}\, u\right)^p$$
and then, as already done for $I_3$, we still find
$$
I_4 \leq c\omega (R^\gamma) \mean{B_{R}}H(x, Du)\, dx\;.
$$
Connecting the estimates found for the terms $I_1, \ldots, I_4$ to \rif{connecting} completes the proof of \rif{comp11}. 
It remains to prove \rif{comp11-ancora}. 
For this we distinguish two cases; the first is when $p\geq 2$; then by Poincar\'e inequality we have 
$$
\notag \mean{B_R} \left|\frac{u-\tilde v}{R}\right|^p \, dx   \leq c\mean{B_R} |Du-D\tilde v|^{p}\, dx \leq  c\mean{B_R}|V_p(Du)-V_p(D\tilde v)|^{2}\, dx \;.
$$
In this case \rif{comp11-ancora} follows combining the inequalities in the last two displays with \rif{comp11}. The case $p<2$ is slightly more elaborate. Using also Young's inequality with conjugate exponents $(2/p, 2/(2-p))$, we have
\begin{eqnarray*}
\nonumber && \hspace{-5mm}\mean{B_R} \left|\frac{u-\tilde v}{R}\right|^p \, dx\leq \mean{B_R} |Du-D  \tilde v|^{p}\, dx \\ \notag&&\stackleq{diffV} c\mean{B_R} \left(|V_p(Du)-V_p(D  \tilde v)|^p (|Du|+|D \tilde v|)^{p(2-p)/2} \right)\, dx
\nonumber 
\\
&&  \, \,  \leq c\left(\mean{B_R} |V_p(Du)-V_p(D \tilde v)|^2\, dx\right)^{p/2}\notag \left(\mean{B_R} (|Du|+|D \tilde v|)^{p} \, dx\right)^{(2-p)/2}\\
&&  \stackleq{infinito} c\left(\mean{B_R} |V_p(Du)-V_p(D \tilde v)|^2\, dx\right)^{p/2}\notag \left(\mean{B_R} [H(x, Du)] \, dx\right)^{(2-p)/2}\\
&&  \ \leq c_\eps \, \mean{B_R} |V_p(Du)-V_p(D \tilde v)|^2\, dx + \eps \mean{B_R} [H(x, Du)] \, dx\;.
\end{eqnarray*}
Once again \rif{comp11-ancora} follows combining the last inequality with \rif{comp11}. 
\end{proof}
To proceed, we gather some regularity information on $\tilde v$. 
\begin{lemma}[Controlled transfer of regularity]\label{auxv} Under the assumptions of Lemma \ref{pqcomp-p}, let $\tilde v  \in W^{1,p}(B_{R})$ be defined as in \trif{var-dir1} and assume that \trif{applica-p} holds in $B_{R}$ with some $s$ as in \trif{esseemme} and some $M\geq 1$.  
Then the following inequality holds for a constant $c\equiv c(\datao,M)$
\eqn{aux-sup}
$$
\sup_{B_{R/4}}\, |\tilde v-(\tilde v)_{B_{R/2}}|^p \leq c \mean{B_{R/2}}|\tilde v-(\tilde v)_{B_{R/2}}|^p\, dx\;,
$$
while the following ones 
\eqn{holder-cont}
$$
[\tilde v]_{0, \gamma;B_{R/2}} \leq c [u]_{0, \gamma;B_{R}}\;, \quad \|H(\cdot, D\tilde v)\|_{L^{1+\delta}(B_{R/2})}\leq c
\|H(\cdot, Du)\|_{L^{1+\delta}(B_{R})}
$$
hold for constants $c\equiv c(\datao)$. In \trif{holder-cont} the exponents $\gamma\equiv \gamma (\datao)$ and $\delta\equiv \delta (\data)$ are those fixed in \trif{stimelocali}. 
\end{lemma}
\begin{proof} Looking at the proof in \cite[Lemma 4.1]{BCM1}, it is easy to see that the $L^\infty$-bound
\eqn{aux-sup2}
$$
\sup_{B_{R/4}}\, |\tilde v|^p \leq c \mean{B_{R/2}}|\tilde v|^p\, dx
$$ holds whenever the Caccioppoli inequality
\eqn{cacc-int} 
$$ \int_{B_{t}} H(x,D(\tilde v-k)_{\pm})\, dx \leq  
c_{cacc} \left(\frac R{s-t}\right)^q\int_{B_{s}} \left|\frac{(\tilde v-k)_\pm}{s-t}\right|^p\, dx  
$$
is satisfied for concentric balls $B_t \Subset B_s\Subset B_{R/2}$ and $\inf \tilde v \leq k \leq \sup \tilde v$; moreover, 
the constant $c$ in \rif{aux-sup2} depends only on $n,p,q$ and $c_{cacc}$. The validity of \rif{cacc-int} follows observing that Lemma \ref{cacc-gen2} applies to $\tilde v$ since this minimizes a functional (the one in \rif{frolast}) that satisfies the assumptions of Lemma \ref{cacc-gen2}. Indeed, here we are assuming \rif{applica-p}. Furthermore, thanks to estimates \rif{infinito}, when applying Lemma \ref{cacc-gen2}, we can dispense from the dependence 
of the various constants on the quantities 
$\|\tilde v\|_{L^{\infty}(B_R)}$ and $\|H(\cdot, D\tilde v)\|_{L^1(B_R)}$. We therefore conclude that \rif{cacc-int} holds for a constant $c_{cacc}$ that ultimately depends only on $\datao$ and $M$ but not on $\tilde v$; from this it follows that also $c$ appearing in 
\rif{aux-sup2} only depends on $\datao$ and $M$. To conclude with \rif{aux-sup}, notice that since 
$\tilde v-(\tilde v)_{B_{R/2}}$ is still a local minimizer of the same functional in \rif{frolast}, then \rif{aux-sup} follows from 
\rif{aux-sup2} applied to $\tilde v-(\tilde v)_{B_{R/2}}$. We then pass to the first estimate in \rif{holder-cont}, arguing in a similar way, but this time applying Theorem \ref{HO} to $\tilde v$; in particular, we need the precise dependence of $\gamma$ stated in \rif{basicexp}. This and the first inequality in \rif{infinito} give that 
$\tilde v \in C^{0, \gamma}_{\loc}(B_R)$ where $\gamma$ is the same as in \rif{stimelocali}. Moreover, estimate \rif{localstimahol} holds for $\tilde v$ 
with $c$ depending only on $\data$; this is still a consequence of \rif{infinito}. 
Applying \rif{localstimahol} to $\tilde v$ and using a standard covering argument, we gain $[\tilde v]_{0, \gamma;B_{R/2}} \leq c (\datao) R^{-\gamma} \osc_{B_{R}} \tilde v$. On the other hand, using also \rif{stimaosc}, it is $\osc_{B_{R}}\, \tilde v\leq \osc_{B_{R}}\, u \leq 2R^\gamma [u]_{0, \gamma;B_{R}}$, so that the first inequality in \rif{holder-cont} follows. It remains to prove the second inequality in \rif{holder-cont} and this time we apply Theorem \ref{HI} to $\tilde v$; again by \rif{infinito} 
the exponent $\delta$ for $v$ does not depend on the solution considered, but only on $\data$, and can be taken to be the same one for 
$u$ fixed in \rif{stimelocali}. Therefore we have
\begin{eqnarray*}
&&\left(\mean{B_{R/2}} [H(x,D\tilde v)]^{1+\delta}\, dx\right)^{1/(1+\delta)} \stackleq{riversa}  c 
\mean{B_{R}} H(x,D\tilde v)\, dx\\ 
&&\qquad \quad  \stackleq{infinito} c\mean{B_{R}} H(x,Du)\, dx \leq c\left(\mean{B_{R}} [H(x,Du)]^{1+\delta}\, dx\right)^{1/(1+\delta)}\;,
\end{eqnarray*}
and the proof is complete. 
\end{proof}
In the following we shall use the classical excess functional given by
\eqn{excess}
$$E\left(w;B_r\right) := \left(\mean{B_{r}} |w - \left(w \right)_{B_{r}} |^p \, dx \right)^{1/p}\;, 
$$
and defined whenever $w \in L^p(B_{r})$ is a possibly vector-valued map and $B_{r} \subset \er^n$ is a ball. We shall several times use the following elementary 
property:
\eqn{elem-prop}
$$
E\left(w;B_r\right) \leq 2\left(\mean{B_{r}} |w - w_0 |^p \, dx \right)^{1/p} \qquad \mbox{for every} \ w_0 \in \er^k\;.
$$
We now go to the main lemma of this section, featuring a first decay estimate:
\begin{lemma}[Blow-up]\label{pcomp} Under the assumptions of Lemma \ref{pqcomp-p}, let $\tilde v  \in W^{1,p}(B_{R})$ be defined as in \trif{var-dir1} and assume also that \trif{applica-p} holds with some $s$ as in \trif{esseemme} and some $M\geq 1$.    
Then, for every choice of $\tilde \eps \in (0,1)$, there exists a positive radius 
\eqn{raggiostella0}
$$R^*\equiv R^*\left(\datao, M,\omega (\cdot), s_m-s, \tilde \eps \right)\;,$$
such that if $R \leq R^*$, then the inequality 
\eqn{stimadec0}
$$\mean{B_{\tau R}} |\tilde v-(\tilde v)_{B_{\tau R}}|^p \, dx\leq c\left\{\tau^p + \tau^{-n} \tilde \eps\right\}\mean{B_{R}}
|\tilde v - (\tilde v)_{B_{R}}|^p\, dx
$$
holds for every $\tau \in (0, 1/16)$, where the constant $c\equiv c(\datao, M)$ is otherwise independent of $\tilde \eps$ and $\tau$. \end{lemma}
\begin{proof} The proof goes in several steps and essentially aims at applying the harmonic approximation argument of Lemma \ref{p-harm} to a suitable blown-up function.  

{\em Step 1: Blow-up}. We start blowing-up $\tilde v$ and $a(\cdot)$ in $B_{R}\equiv B_R(x_0)$, thereby defining 
\eqn{defiv}
$$v(x) := \frac{ \tilde v(x_0+Rx)-(\tilde v)_{B_{R/2}}}{E\left(\tilde v;B_{R/2}\right)}\quad \mbox{and}\quad \ a_R(x):=
a\left(x_0+Rx\right)\;, \quad x\in B_1\;.
$$
Notice that we can always assume that $E\left(\tilde v;B_{R/2}\right)>0$, otherwise \rif{stimadec0} follows trivially.
Abbreviating $E(R)\equiv E\left(\tilde v;B_{R/2}\right)$, by \rif{applica-p}, the Caccioppoli inequality in \rif{cacc1} applies to $\tilde v$ in the ball $B_{R/2}$ as well. The result is that 
$$
\mean{B_{R/4}} H(x,D\tilde v)\, dx \leq c\mean{B_{R/2}}\left|\frac{\tilde v-(\tilde v)_{B_{R/2}}}{R}\right|^p \, dx
$$ 
holds, where $c\equiv c(\datao,M)$. Notice that, as already done for Lemma \ref{auxv},  here we are using \rif{infinito} to dispense from the dependence on $\|\tilde v\|_{L^\infty(B_R)}$ and incorporate it in the one on $\|u\|_{L^\infty(\Omega_0)}$ and, ultimately, in $\datao$. In view of the definition in \rif{excess}, the inequality in the above display can be re-written as 
$$
\mean{B_{R/4}} \left([R/E(R)]^p |D\tilde v|^p + [R/E(R)]^pa(x)|D\tilde v|^q\right)\, dx\leq c 
$$
that in terms of $v$ means
\eqn{stima-ass}
$$
\mean{B_{1/4}} \left(|Dv|^p + [E(R)/R]^{q-p}a_R(x)|Dv|^q\right)\, dx\leq c\;,
$$
again for $c\equiv c(\datao,M)$. It is now easy to see that $v$ is a local minimizer of the functional defined by
\eqn{rescaleF}
$$
w \mapsto \int_{B_1} \bar F(x, Dw)\, dx$$ with
\eqn{rescaleF2} 
$$\bar F(x, z):= \frac{F\left(x_0+Rx, (u)_{B_{R}}, [E(R)/R]z\right)}{\left[E(R)/R\right]^p}\;, 
$$
for every $x \in B_1$ and $z \in \er^n$, 
and, as such, it satisfies its Euler-Lagrange equation. This means that
\eqn{eeul}
$$
\mean{B_{1}} \big\langle  \partial F(x_0+Rx, (u)_{B_{R}}, [E(R)/R] D v), D\varphi \big\rangle \, dx=0 $$ 
holds for every $\varphi \in W^{1,\infty}_0(B_{1})$. 

{\em Step 2: Universal energy bounds.} We check the properties of the rescaled integrand $\bar F(\cdot)$ in \rif{rescaleF2}. By defining the new control function
\eqn{nuovaH}
$$
\bar H(x,z):= |z|^p +  [E(R)/R]^{q-p}a_R(x)|z|^q\;.
$$
the following conditions are now satisfied whenever $x\in B_1$, and $z \in \er^n$:
\eqn{assFdd}
$$
\nu \bar H(x,z)\leq \bar F(x,z) \leq L\bar H(x, z)$$
 as a direct consequence of \rif{growtH} and \rif{rescaleF2}. We now want to check that $\bar F(\cdot)$ satisfies the conditions allowing to apply Theorem \ref{HI} in $B_{1/4}$, this time to $v$. For this we have to check that $v$ is bounded in $B_{1/4}$ (in fact needed only when $p<n$) and that, in every case, the new coefficient 
$x \mapsto [E(R)/R]^{q-p}a_R(x)$ satisfies the needed H\"older continuity assumption in $B_{1/4}$, that is
\eqn{hol-res}
$$
 [E(R)/R]^{q-p}[a_R]_{0, \alpha;B_{1/4}} \leq c(\datao)\;.
$$
Recalling the definition of $v$ in \rif{defiv}, and using Lemma \ref{auxv}, \rif{aux-sup} yields
\eqn{boundtv}
$$
\|v\|_{L^{\infty}(B_{1/4})}\leq c(\datao,M)\;. 
$$
For \rif{hol-res}, we distinguish two cases. 

{\em Case assumption \trif{bound2} is in force.} By \rif{stimelocali} and \rif{holder-cont} we have 
\eqn{be}
$$E(R) =\left(\mean{B_{R/2}} |\tilde v - \left(\tilde v \right)_{B_{R/2}} |^p \, dx \right)^{1/p}\leq cR^{\gamma}$$ for a constant $c\equiv c(\datao)$. 
To prove \rif{hol-res}, observe that, whenever $x_1, x_2 \in B_1$, it follows that
\begin{eqnarray}
 && \notag \left|[E(R)/R]^{q-p}a_R(x_1)-[E(R)/R]^{q-p}a_R(x_2)\right| \\ 
 &&\notag  \qquad \leq  c[a]_{0, \alpha}[E(R)/R]^{q-p} 
R^{\alpha}|x_1-x_2|^\alpha \\ && \quad \  \stackleq{be}  c R^{(p-q)(1-\gamma)+\alpha}|x_1-x_2|^\alpha \stackleq{bound2} c (\datao) |x_1-x_2|^\alpha\;. \label{evident}
\end{eqnarray}
\begin{remark}\label{stupido}
Notice that, in order to perform the last estimation, we essentially need the larger bound $q\leq  p+\alpha/(1-\gamma)$ related to the one appearing in \rif{bound3}. 
\end{remark}

{\em Case assumption \trif{bound1} is in force and $p\geq n$.} Recalling \rif{stimelocali} and the second inequality in \rif{holder-cont}, by Sobolev-Morrey embedding theorem we then have
\begin{eqnarray}
\notag E(R) &\leq & c(n,p, \delta) \|D\tilde v\|_{L^{p(1+\delta)}(B_{R/4})}R^{1-\frac{n}{p(1+\delta)}} \\ &\stackleq{holder-cont} &  
c(\datao)\|H(\cdot, Du)\|_{L^{1+\delta}(B_{R/2})}^{1/p}R^{1-\frac{n}{p(1+\delta)}}\notag \\ &\stackleq{stimelocali} & c(\datao)R^{1-\frac np + \frac{n\delta}{p(1+\delta)}}.\label{sob-mor}
\end{eqnarray}
Therefore this time we can estimate, again for $x_1, x_2 \in B_1$
\begin{eqnarray*}
\left|[E(R)/R]^{q-p}a_R(x_1)-[E(R)/R]^{q-p}a_R(x_2)\right| &\leq & c[a]_{0, \alpha}R^{n(1-q/p)+\alpha} |x_1-x_2|^\alpha\\
&\leq & c(\datao)|x_1-x_2|^\alpha \;,
\end{eqnarray*}
as now it is $q/p \leq 1+\alpha/n$, and \rif{hol-res} follows in this case too. 

With \rif{hol-res} and \rif{boundtv} at our disposal, and recalling that \rif{assFdd} holds, we are now able to apply Theorem \ref{HI} to $v$, thereby obtaining the existence of a higher integrability exponent 
$\delta_1 >0$ and a constant $c$, both depending on $\datao$ and $M$, but otherwise independent of $R$, such that
\eqn{nuova-maggiore}
$$
\left(\mean{B_{1/8}} [\bar H(x,Dv)]^{1+\delta_1}\, dx\right)^{1/(1+\delta_1)} \leq c 
\mean{B_{1/4}} \bar H(x,Dv)\, dx\;.
$$ 
Moreover, recalling \rif{stima-ass}, we conclude with
\eqn{appi1}
$$
\mean{B_{1/4}} \bar H(x,Dv)\, dx+ \mean{B_{1/8}} [\bar H(x,Dv)]^{1+\delta_1}\, dx \leq c(\datao,M)\;.
$$ 

{\em Step 3: Frozen functional.} Let us now define 
$$\bar F_0(z):= \frac{F\left(x_B, (u)_{B_{R}}, [E(R)/R]z\right)}{[E(R)/R]^{p}}
$$
and the related control Young function 
\eqn{ultimaH}
$$
\bar H_0(z):= |z|^p +  [E(R)/R]^{q-p}a_{{\rm i}}(B_{R})|z|^q
$$
for every $z \in \er^n$. 
Later on, we shall use the functional 
\eqn{frozenfun}
$$
W^{1,\bar H_0}(B_{1/8}) \ni w \mapsto \int_{B_{1/8}} \bar F_0(Dw)\, dx\;.
$$
We notice that the integrand $\bar F_0(\cdot)$ satisfies the 
following growth and ellipticity conditions:
\eqn{assFhh}
$$
\left\{
\begin{array}{c}
\nu \bar H_0(z)\leq \bar F_0(z) \leq L\bar H_0(z)\\ [8 pt]
|\partial \bar F_0(z)||z|+|\partial^2 \bar F_0(z)||z|^2 \leq L\bar H_0(z)\\ [6 pt]
\displaystyle \nu \frac{\bar H_0(z)}{|z|^2}|\xi|^2\leq  \big\langle \partial^2 \bar F_0(z)\xi, \xi\big\rangle \;, 
 \end{array}\right.
$$
for every choice of $z\in \er^n\setminus \{0\}$ and $\xi \in \er^n$. 

{\em Step 4: Harmonic type approximation.}
We notice that, with $\varphi \in W^{1,\infty}_0(B_{1/8})$ being such that $\|D\varphi\|_{L^\infty(B_{1/8})}\leq 1$, by \rif{eeul} we can write
\begin{eqnarray}
\notag&&\left|\mean{B_{1/8}} \Big\langle \partial \bar F_0 (D v), D\varphi  \Big\rangle\, dx \right|\\
\notag&&   = \left|\mean{B_{1/8}} \Big\langle \partial \bar F_0 (D v)-\frac{\partial F(x_0+Rx, (u)_{B_{R}}, [E(R)/R] D v)}{[E(R)/R]^{p-1}}, D\varphi  \Big\rangle \, dx\right|\\
&&   \leq \mean{B_{1/8}} \left| \partial \bar F_0 (D v)-\frac{\partial F(x_0+Rx, (u)_{B_{R}}, [E(R)/R] D v)}{[E(R)/R]^{p-1}} \right|\, dx :=\mathcal I\;.\label{iltermine}
\end{eqnarray}
The term $\mathcal I$ can be estimated using \rif{assF}$_3$. Recalling that $a_{{\rm i}}(B_R)\leq a_R(x)$ for every $x \in B_1$, this gives
\begin{eqnarray}
\notag \mathcal I &\leq & c\omega(R) \mean{B_{1/8}} |Dv|^{p-1}\, dx +c\omega(R) 
\left[\frac{E(R)}{R}\right]^{q-p}\mean{B_{1/8}} a_R(x)|Dv|^{q-1}\, dx\\
\notag 
&& \quad + c
\left[\frac{E(R)}{R}\right]^{q-p}\mean{B_{1/8}} |a_R(x)-a_{{\rm i}}(B_R)||Dv|^{q-1}\, dx\\
&=:& \mathcal I_1+\mathcal I_2+\mathcal I_3\;, \label{merging}
\end{eqnarray}
with $c\equiv c(p,q,L)$. We then estimate separately the above three terms. In any case we have, by H\"older's inequality, we have that 
$$\mathcal I_1\leq c\omega(R)\left(\mean{B_{1/8}} |Dv|^{p}\, dx\right)^{1-1/p}\stackleq{stima-ass} c\, \omega (R)
$$
for $c\equiv c (\datao, M)$. For the remaining terms we distinguish between two different cases, as already done Step 2.

{\em Case assumption \trif{bound2} is in force.}
We have 
\begin{eqnarray}
\notag && \mathcal I_2\\
&& \quad = c\, \omega (R)\left[\frac{E(R)}{R}\right]^{\frac{q-p}{q}}
\notag \mean{B_{1/8}} [a_R(x)]^{\frac1q}[a_R(x)]^{\frac{q-1}{q}}\left[\frac{E(R)}{R}\right]^{\frac{(q-p)(q-1)}{q}}|Dv|^{q-1}\, dx\\ 
\notag &&\ \stackleq{be}   c\, \omega (R)M^{1/q}R^{\frac{p+\alpha -q + \gamma(q-p)-s}{q}} 
\left(\mean{B_{1/8}} [E(R)/R]^{q-p}a_R(x)|Dv|^q\, dx\right)^{1-1/q}\\
&& \ \stackleq{stima-ass}   c\omega(R)\;, \label{stimaI2}
\end{eqnarray}
for $c\equiv c(\datao,M)$, where we have used $R\leq 1$ and \rif{esseemme}. Similarly, we have
\begin{eqnarray}
\notag \mathcal I_3&\leq &  c [a]_{0, \alpha}^{1/q}\left[\frac{E(R)}{R}\right]^{\frac{q-p}{q}}R^{\frac{\alpha}{q}}
\notag \mean{B_{1/8}}[a_R(x)]^{\frac{q-1}{q}}\left[\frac{E(R)}{R}\right]^{\frac{(q-p)(q-1)}{q}}|Dv|^{q-1}\, dx\\ 
\notag & \stackleq{be} &   cR^{\frac{p+\alpha -q + \gamma(q-p)-s}{q}} 
\left(\mean{B_{1/8}} [E(R)/R]^{q-p}a_R(x)|Dv|^q\, dx\right)^{1-1/q}\\
&\stackleq{stima-ass} &  cR^{\frac{p+\alpha -q + \gamma(q-p)-s}{q}}  \;.  \label{stimaI3}
\end{eqnarray}
Collecting the estimates found for the terms $\mathcal I_1, \mathcal I_2$ and $\mathcal I_3$ in the last three displays and merging them with \rif{merging}, we conclude with
\eqn{primaI}
$$
\mathcal I \leq c(\datao,M) \left[\omega(R)+R^{\frac{p+\alpha -q + \gamma(q-p)-s}{q}} \right]\;.
$$

{\em Case assumption \trif{bound1} is in force and $p\geq n$.} Re-writing as in \rif{stimaI2}, but using this time 
\rif{sob-mor}, we have
\begin{eqnarray}
\notag &&\left[\frac{E(R)}{R}\right]^{q-p}\mean{B_{1/8}} a_R(x)|Dv|^{q-1}\, dx\\
\notag
&&\ \quad \stackleq{condiadopo}cM^{1/q}
R^{\frac{\alpha-s}{q}}\left[\frac{E(R)}{R}\right]^{\frac{q-p}{q}} \left(\mean{B_{1/8}}  [E(R)/R]^{q-p}
a_R(x)|Dv|^{q}\, dx\right)^{1-1/q}\\
\notag
&&\ \quad \stackleq{stima-ass}cM^{1/q}
R^{\frac{\alpha-s}{q}}\left[\frac{E(R)}{R}\right]^{\frac{q-p}{q}} \\ &&\ \quad \stackleq{sob-mor}  
c(\datao,M)R^{\frac{1}{q}\left[\alpha -\frac {(q-p)n}p +\frac{n\delta(q-p)}{p(1+\delta)}-s\right]}\;.\label{gen-est}
\end{eqnarray}
Notice that the exponent of $R$ in the last line of \rif{gen-est} is non-negative by \rif{bound1} and \rif{esseemme}. By using \rif{gen-est} we readily have that 
$$
\mathcal I_2\leq cR^{\frac{1}{q}\left[\alpha -\frac {(q-p)n}p +\frac{n\delta(q-p)}{p(1+\delta)}-s\right]}\,\omega(R)\leq c\,  \omega(R)\;,
$$
again for $c\equiv c(\datao,M)$, 
while, recalling the definition of $a_{{\rm i}}(B_R)$ in \rif{applica0}, we finally estimate
$$
\mathcal I_3\leq c\left[\frac{E(R)}{R}\right]^{q-p}\mean{B_{1/8}} a_R(x)|Dv|^{q-1}\, dx\leq c R^{\frac{1}{q}\left[\alpha -\frac {(q-p)n}p +\frac{n\delta(q-p)}{p(1+\delta)}-s\right]}\;, 
$$ 
again with $c\equiv c(\datao,M)$. Merging the estimates for the three terms $\mathcal I_1, \mathcal I_2$ and $\mathcal I_3$ with the one in \rif{merging}, we this time get
\eqn{primaI2}
$$
\mathcal I \leq c(\datao,M) \left\{\omega(R)+R^{\frac{1}{q}\left[\alpha -\frac {(q-p)n}p +\frac{n\delta(q-p)}{p(1+\delta)}-s\right]} \right\}\;.
$$
\begin{remark}\label{stupido2} The analysis of the first case is still valid when assuming larger bound in \rif{bound3}. This is better than the one in \rif{bound1} if 
$
p > n/(1-\gamma). 
$
\end{remark}
Now, taking into account the estimates found for $\mathcal I$ in \rif{primaI} and \rif{primaI2}, and recalling \rif{iltermine}, by finally defining 
\eqn{defioR}
$$
\texttt{o}(R):=
\left\{
\begin{array}{ccc}
\omega(R)+R^{\frac{1}{q}\left[\alpha -\frac {(q-p)n}p +\frac{n\delta(q-p)}{p(1+\delta)}-s\right]}&\mbox{if we use \trif{bound1}}& \\ [7 pt]
\omega(R)+R^{\frac{p+\alpha -q + \gamma(q-p)-s}{q}}&\mbox{if we use \trif{bound2}}\;,
\end{array} \right.
$$
which is non-decreasing, we conclude with
\eqn{appi2}
$$
\left|\mean{B_{1/8}} \big\langle \partial \bar F_0 (D v), D\varphi  \big\rangle\, dx \right|\leq c_h\,\texttt{o}(R)\|D\varphi\|_{L^{\infty}}$$
that now holds for every $\varphi \in W^{1,\infty}_0(B_{1/8})$, where $c_h\equiv c_h(\datao,M)$. 
Observe also that, with the notation in \rif{esseemme}, we have in any case - either we use \rif{bound1} or \rif{bound2} - that
\eqn{precisaO} 
$$
\texttt{o}(R)= \omega(R) + R^{(s_m-s)/q}\;.
$$
By \rif{appi1}, \rif{appi2} and \rif{assFhh} we are now in position to apply Lemma~\ref{p-harm} 
with the choice 
$
A_0(z) \equiv \partial \bar F_0(z)$, $H_0(z)\equiv \bar H_0(z)$, and $a_0 \equiv  [E(R)/R]^{q-p}a_{{\rm i}}(B_{R})$, 
so that, in particular, assumptions \rif{ass0} are satisfied. 
By Lemma~\ref{p-harm} there exists a function 
$h \in v+ W^{1,\bar H_0}_0(B_{1/8})$ such that
\begin{equation}\label{p-har-dopo}
\mean{B_{1/8}}\big\langle \partial  \bar F_0(Dh), D\varphi \big\rangle\, dx=0\qquad\qquad \text{for all $\varphi \in W^{1,\infty}_0(B_{1/8})$}\,,
\end{equation}
\eqn{stimaHH}
$$
\mean{B_{1/8}} \bar H_0(Dh)\, dx+ \mean{B_{1/8}}\big[H_0(Dh)\big]^{1+\delta_0} \, dx \leq c\, ,
$$
\begin{eqnarray}
&&\notag  \mean{B_{1/8}}\left( |V_p(Dv)-V_p(Dh)|^2+[E(R)/R]^{q-p}a_{{\rm i}}(B_{R})|V_q(Dv)-V_q(Dh)|^2\right) \, dx \\
&& \quad  \leq c\, [\texttt{o}(R)]^{s_1}\label{stimadiff}
\end{eqnarray}
and finally
\eqn{stimadiff-final-pre}
$$
\mean{B_{1/8}} \left(|v-h|^p+[E(R)/R]^{q-p}a_{{\rm i}}(B_{R})|v-h|^q\right)\, dx \leq c_d\, [\texttt{o}(R)]^{s_0}\;.
$$
In all the estimates above the constants $c, c_d \geq 1$ and $ s_0, s_1\in (0, 1)$ depend on $\datao$ and $M$, but are otherwise independent of $R$. 

{\em Step 5: Choice of the radius $R_*$ in \trif{raggiostella0}}. Given the definition of $\texttt{o}(R)$ in \rif{precisaO}, and given $\tilde \eps \in (0,1)$ as in the statement, we determine $R_*$ small enough to get
$$
c_d[\omega(R_*)]^{s_0} + c_dR_*^{s_0(s_m-s)/q}\leq \tilde \eps\;.
$$
This, given the dependence of $s_0$ and $c_d$ on $\datao$ and $M$, gives the final dependence displayed in \rif{raggiostella0}. By \rif{stimadiff-final-pre} we conclude that 
\eqn{stimadiff-final}
$$
\mean{B_{1/8}} \left(|v-h|^p+[E(R)/R]^{q-p}a_{{\rm i}}(B_{R})|v-h|^q\right)\, dx \leq \tilde \eps\;.
$$
\indent {\em Step 6: Proof of the decay estimate \trif{stimadec0}}. We first observe that
by a standard density argument the relation in \rif{p-har-dopo} 
continues to hold whenever $\varphi \in W^{1,1}_0(B_{1/8})$ is such that 
$\bar H_0(D\varphi) \in L^{1}(B_{1/8})$. This means that $h$ is a local minimizer
of the functional
$$
W^{1,\bar H_0}(B_{1/8}) \ni w \mapsto \int_{B_{1/8}} \bar F_0(Dw)\, dx\;.$$
Conditions \rif{assFhh} are satisfied by the integrand $\bar F_0(\cdot)$, so we are able to apply the results from \cite{Lieb} (see also \cite{Baroni}) to deduce the following a priori gradient bound:
\eqn{stima-lieb1}
$$
 \sup_{B_{1/16}}\,  \bar H_0( Dh)  \leq c  \mean{B_{1/8}}\bar H_0( Dh) \, dx\;,
$$
where this time $c\equiv c (n,p, q,\ratio)$. By taking $\tau \leq 1/16$, we can estimate
\begin{eqnarray*}
\mean{B_{\tau}} |v-(v)_{B_{\tau}}|^p \, dx & \leq &c \mean{B_\tau} |v -(h)_{B_\tau} |^p \, dx \\ 
&\leq &c \mean{B_\tau} |h -(h)_{B_\tau} |^p \, dx+c\tau^{-n}\mean{B_{1/8}} |v -h |^p \, dx 
\\ & \stackleq{stimadiff-final} &c\tau^p \sup_{B_\tau} \, |Dh|^p + c\tau^{-n} \tilde \eps
\\ & \leq &c\tau^p \sup_{B_\tau} \, \bar H_0( Dh) + c\tau^{-n} \tilde \eps
\\ & \stackleq{stima-lieb1} & c\tau^p \mean{B_{1/8}}\bar H_0( Dh) \, dx+ c\tau^{-n} \tilde \eps
\\ & \stackleq{stimaHH} &c\tau^p + c\tau^{-n} \tilde \eps\;. 
\end{eqnarray*}
By scaling back to $\tilde v$ and recalling \rif{defiv}, we get \rif{stimadec0} after some elementary manipulations. 
\end{proof}
A crucial outgrow of the previous proof is 
\begin{lemma}[Comparison]\label{pcomp-dopo} In the setting of Lemma \ref{pqcomp-p}, let $\tilde v  \in W^{1,p}(B_{R})$ be defined as in \trif{var-dir1} and assume also that \trif{applica-p} holds with some $s$ as in \trif{esseemme} and some $M\geq 1$. Then, there exists a function $\tilde h \in \tilde v +W^{1,p}_0(B_{R/8})$ which is a local minimizer of the functional 
\eqn{frolast22}
$$
\mathcal F_{B}(w):=   \int_{B_{R/8}} F(x_B, (u)_{B_{R}}, Dw) \, dx
$$
in the sense of Definition \ref{defimain} and that satisfies 
\begin{eqnarray}
\nonumber &&\mean{B_{R/8}} \left( |V_p(Du)-V_p(D\tilde h)|^2+ a_{{\rm i}}(B_{R})|V_q(Du)-V_q(D\tilde h)|^2\right)\, dx\\
&&\quad  \leq c\left\{ [\omega(R^\gamma)]^{s_1}+ R^{s_1(s_m-s)/q} \right\}\mean{B_{R}} H(x, Du) \, dx
\label{comp11dd}
\end{eqnarray}
for a constant $c$ and a positive exponent $s_1$, both depending only on $\datao$ and $M$. Finally, the inequality
\eqn{stima-energia}
$$
\mean{B_{R/8}} \left( |D\tilde h|^p+ a_{{\rm i}}(B_{R})|D\tilde h|^q\right)\, dx \leq c
\mean{B_{R}} H(x, Du)\, dx 
$$
holds for $c\equiv c (n, \ratio)$. 
\end{lemma}
\begin{proof}We go back to the proof of Lemma \ref{pcomp}, Steps 3 and 4 and consider the function $h \in v+ W^{1,\bar H_0}_0(B_{1/8})$ found by the application of Lemma \ref{p-harm} and solving \rif{p-har-dopo}. We blow-down $h$ defining \eqn{nuovah}
$$
\tilde h(x):= E\left(\tilde v;B_R\right)\,h\left(\frac{x-x_0}{R}\right)\,,\qquad x \in B_{R/8}(x_0)\;.
$$
We notice that $\tilde h \in \tilde v + W^{1,\tilde H_0}_0(B_{R/8})$, where now 
$\tilde H_0(z):= |z|^p +  \ai(B_{R})|z|^q$. The function $h$ solves \rif{p-har-dopo}, which is the Euler-Lagrange equation of the functional in \rif{frozenfun}, 
and therefore, by strong convexity, a local minimizer too. More precisely, we have that 
$$
\int_{B_{1/8}} \bar F_0(Dh) \, dx \leq \int_{B_{1/8}} \bar F_0(D h+D\varphi) \, dx
$$
holds whenever $\varphi\in W^{1,\bar H_0}_0(B_{1/8})$ and $\bar H_0(\cdot)$ is defined as in \rif{ultimaH}. This implies that $\tilde h$ in \rif{nuovah} 
minimizes the functional in \rif{frolast22} in the sense that 
$
\mathcal F_B(\tilde h) \leq \mathcal F_B(\tilde h+\varphi )$ holds for every $\varphi\in W^{1, \tilde H_0}_0(B_{R/8})$. By \rif{assFhh} $\tilde h$ satisfies the energy estimate 
\eqn{energy-estimate-ultima}
$$
\mean{B_{R/8}} \tilde H_0(D\tilde h)\, dx \leq \frac L\nu
\mean{B_{R/8}}\tilde H_0(D\tilde v)\, dx \leq \frac {8^nL}\nu\mean{B_{R}}H(x,D\tilde v)\, dx
$$
which can be obtained as in \rif{coerc}. Given the definitions of $\tilde h$ and of $v$ in \rif{defiv},
the comparison estimate in \rif{stimadiff} gives, using also Poincar\'e inequality and \rif{elem-prop}
\begin{eqnarray}
&&\notag  \mean{B_{R/8}}\left( |V_p(D\tilde v)-V_p(D\tilde h)|^2
+a_{{\rm i}}(B_{R})|V_q(D\tilde v)-V_q(D\tilde h)|^2\right) \, dx \\
&& \qquad   \leq c\, [\texttt{o}(R)]^{s_1}\mean{B_{R}}
\left|\frac{\tilde v - (\tilde v)_{B_{R}}}{R}\right|^p\, dx
\leq c\,  [\texttt{o}(R)]^{s_1}\mean{B_{R}}
|D\tilde v|^p\, dx \notag \\
&& \qquad  \leq c\, [\texttt{o}(R)]^{s_1}\mean{B_{R}}
H(x, D\tilde v)\, dx
\stackleq{infinito} c\,  [\texttt{o}(R)]^{s_1}\mean{B_{R}}
H(x, Du)\, dx\;. \label{lasti}
\end{eqnarray}
Here it is $c \equiv c(\datao,M)$ while $\texttt{o}(R)$ has been defined in \rif{defioR} and \rif{precisaO}. Combining \rif{lasti} with \rif{comp11}, and recalling again \rif{precisaO}, we get \rif{comp11dd} again after a few elementary manipulations. Finally, \rif{stima-energia} follows using \rif{energy-estimate-ultima} together with the second inequality in \rif{infinito}. 
\end{proof}
We finally conclude this section with another decay estimate.  
\begin{lemma}\label{pcomp3p} Let $ u \in W^{1,p}(\Omega)$ be a local minimiser of the functional $\mathcal F$ defined in \trif{funF}, under the assumptions \trif{assF}-\trif{convex-omega}. Assume that either \trif{bound1} or \trif{bound2} 
is satisfied, and furthermore that 
\eqn{applica-pt}
$$
a_{{\rm i}}(B_{\tau R}) \leq M[a]_{0,\alpha}(\tau R)^{\alpha-s}
$$ 
holds with some $s$ as in \trif{esseemme} and some $M\geq 1$. Then, for every choice of $\bar \eps \in (0,1)$, there exists a positive radius $R_*$, \eqn{raggiostella}
$$R_*\equiv R_*(\datao, M, \omega (\cdot), s_m-s, \bar \eps)\;,$$
such that if $R \leq R_*$, then the inequality 
\eqn{seconda-p1} 
$$\int_{B_{\tau R}}H(x, Du)\, dx  \leq c_p \left\{\tau^{n} + c_\eps \tau^{-p} \bar \eps + \tau^{-p}\eps \right\} \int_{B_R}H(x, Du)\, dx\;,$$
holds for every $\tau \in (0, 1/32)$ and every $\eps \in (0,1)$, where the constants $c_p\equiv c_p(\datao, M)$ and $c_\eps\equiv c_\eps (\eps, p,q)$ are otherwise independent of $\tilde \eps $ and $\tau$. 
\end{lemma}
\begin{proof} We notice that \rif{applica-pt} implies that also \rif{applica-p} holds, that is,
\eqn{holdstrue}
 $$a_{{\rm i}}(B_{ R})\linebreak  \leq M[a]_{0,\alpha}R^{\alpha-s}$$
holds true. Indeed, on the contrary, thanks to \rif{ognir}, $a_{{\rm i}}(B_{ R}) > M[a]_{0,\alpha}R^{\alpha-s}$ 
would imply that  
$a_{{\rm i}}(B_{\tau R}) > M[a]_{0,\alpha}(\tau R)^{\alpha-s}$, that is, a contradiction \rif{applica-pt}. The validity of \rif{holdstrue}  means we can apply Lemma \ref{pcomp} with $\tilde \eps\equiv \eps_1$ to be chosen in a few lines, and we can use \rif{stimadec0} provided $R \leq R^*\equiv R^*(\datao, M, \omega(\cdot), s_m-s,\eps_1)$ is determined via \rif{raggiostella0}. We therefore get, with some elementary manipulations and using \rif{elem-prop} repeatedly 
\begin{eqnarray}
\notag && \mean{B_{\tau R}} |u-(u)_{B_{\tau R}}|^p \, dx  \leq  2^p \mean{B_{\tau R}} |u-(\tilde v)_{B_{\tau R}}|^p \, dx\\ \notag && \qquad \leq c\mean{B_{\tau R}} |\tilde v-(\tilde v)_{B_{\tau R}}|^p \, dx + c\tau^{-n}\mean{B_{ R}} |u-\tilde v|^p \, dx\\
&& \qquad  \leq c\left\{\tau^p + \tau^{-n} \eps_1 \right\}\mean{B_{R}}
|u - (u)_{B_{R}}|^p\, dx   + c\tau^{-n}\mean{B_{ R}} |u-\tilde v|^p \, dx\;.
\label{lasty}
\end{eqnarray} 
Using \rif{comp11-ancora} to estimate the last integral in the above display, and performing some other standard manipulations, we further get
\begin{eqnarray*}
&&\int_{B_{\tau R}} \left|\frac{u-(u)_{B_{\tau R}}}{\tau R}\right|^p \, dx\notag 
\\&& \quad  \leq c\left\{\tau^{n} + c_\eps \tau^{-p}  \omega(R^\gamma)+ \tau^{-p}\eps_1 + \tau^{-p}\eps \right\}\int_{B_{R}}
H(x, Du)\, dx\;.
\end{eqnarray*}
This holds for every $\tau \in (0, 1/16)$ and $c\equiv c (\datao, M)$. Then we select $\eps_1=\bar \eps/2$ and finally, we take $R_*\leq R^*$ such that $\omega(R_*^{\gamma})\leq \bar \eps/2$. This choice determines the dependence described in \rif{raggiostella} and yields
\eqn{stimadec0-dopo}
$$
\int_{B_{\tau R}} \left|\frac{u-(u)_{B_{\tau R}}}{\tau R}\right|^p \, dx 
\leq c\left\{\tau^{n} + c_\eps \tau^{-p}  \bar \eps+ \tau^{-p}\eps \right\}\int_{B_{R}}
H(x, Du)\, dx\;.
$$
On the other hand 
the validity \rif{applica-pt} implies that applicability conditions of Lemma \ref{cacc-gen2} 
are verified with with $B_R$ 
replaced by $B_{\tau R}$. We therefore deduce
$$\int_{B_{\tau R/2}}H(x, Du)\, dx \leq c\int_{B_{\tau R}} \left|\frac{u-(u)_{B_{\tau R}}}{\tau R}\right|^p \, dx\;.$$
Using this last inequality together with \rif{stimadec0-dopo} (and renaming $\tau/2$ in $\tau$) yields \rif{seconda-p1} and the proof is complete.\end{proof}

 
\section{$(p,q)$-phase}\label{fasepq}
In this section we consider the condition which is complementary to  \rif{applica-p}, that is 
\rif{applica}. This time we consider it with {\em the special choice $s=0$}, i.e., 
\eqn{applica-real}
$$
a_{{\rm i}}(B_{R}) > M[a]_{0,\alpha}R^{\alpha}\qquad \quad  \mbox{for some}\ M\geq 1\;.$$ 
Notice that we can actually always reduce to this case since \rif{applica} implies \rif{applica-real} for every $s $ as in \rif{esseemme} (recall that we always take $R\leq 1$). 
We consider the different frozen functional 
\eqn{frolast2}
$$
\mathcal F_{B}(w):=   \int_{B_{R}} F(x_B, (u)_{B_{R}}, Dw) \, dx\;,$$
and define $h \in u + W^{1,p}_0(B_{R})$ as the solution to the following Dirichlet problem:
\eqn{var-dir2}
$$
\left\{
    \begin{array}{c}
   \displaystyle  \tilde h \mapsto \min_w
\, \mathcal F_{B}(w)    
\\[8pt]
        w \in u + W^{1,p}_0(B_{R})\;. 
\end{array}\right.
$$
As for \rif{coerc}, we get
\eqn{finito2}
$$ \int_{B_{R}} \left(|D \tilde h|^p+a_{{\rm i}}(B_{R})|D \tilde h|^q\right) \, dx \leq \frac{L}{\nu}\int_{B_{R}} \left(|Du|^p+a_{{\rm i}}(B_{R})|Du|^q\right) \, dx\;.$$
Notice that now it is $\tilde h \in W^{1,q}(B_{R})$. We then have the following direct analog of Lemma \ref{pcomp-dopo}. 
\begin{lemma}\label{pqcomp} Let $ u \in W^{1,p}(\Omega)$ be a local minimiser of the functional $\mathcal F$ defined in \trif{funF}, under the assumptions \trif{growtH}, \trif{assF} and \trif{convex-omega-2}. Assume that either \trif{bound1} or \trif{bound2} is satisfied. Let $B_{R}\Subset \Omega_0\Subset \Omega$ be a ball with $R\leq 1$ and such that the condition \trif{applica-real} is in force. 
Then there exists a positive constant $c\equiv c(\datao)$, which is independent of $M$, such that the following estimate holds:
\begin{eqnarray}
\nonumber &&\mean{B_{R}} \left( |V_p(Du)-V_p(D \tilde h )|^2+ a(x) |V_q(Du)-V_q(D \tilde h )|^2\right)\, dx\\
&&\quad  \leq c\left[\omega(R^\gamma) + \frac 1M\right] \mean{B_{R}} H(x, Du) \, dx\;.
\label{comp11-pq}
\end{eqnarray}
Here $\gamma\equiv \gamma (\datao)\in (0,1)$ is the H\"older continuity exponent defined in Theorem \ref{HO}. 
Moreover, it holds that
\eqn{finito3}
$$ \int_{B_{R}} H(x,D \tilde h) \, dx \leq \frac{3L}{\nu}\int_{B_{R}} H(x,Du) \, dx\;.$$
\end{lemma}
\begin{proof} First, let us observe that assumptions 
\rif{assF}
also imply that 
\begin{eqnarray}
&&\notag | F(x_1,v,z)-F(x_2,v,z)|\\
&& \qquad \leq L\, \omega (|x_1-x_2|)\left[H(x_1,z)+H(x_2,z) \right] + L|a(x_1)-a(x_2)||z|^{q}
\label{ulteriore}
 \end{eqnarray}
holds for every choice of $x_1, x_2 \in \Omega$, $v \in \er$ and $z \in \er^n$. 
Indeed, we can write 
\begin{eqnarray*}
|F(x_1,v,z)-F(x_2,v,z)| &=& \left| \int_0^1 \big\langle \partial F (x_1,v,tz) -\partial F (x_2,v,tz), z\big\rangle\,  dt \right|\\
&\leq& \int_0^1|\partial F (x_1,v,tz) -\partial F (x_2,v,tz)|\,  dt\,|z|\;,
\end{eqnarray*}
so that \rif{ulteriore} follows by \rif{assF}. 
Notice that, in order to derive the equality in the above display, we have used that $F(x_1,v,0)=F(x_2,v,0)= 0$, which is in turn implied by \rif{growtH}. Next, observe the basic consequence of 
\rif{applica}
\eqn{thismeans}
$$
a(x)\leq a(x)-a_{{\rm i}}(B_R) + a_{{\rm i}}(B_R) \leq 2[a]_{0, \alpha}R^\alpha + a_{{\rm i}}(B_R)\leq 3 a_{{\rm i}}(B_R)\leq 3 a(x)\;,
$$
that holds for every $x \in B_R$. 
This and \rif{finito2} easily imply \rif{finito3}. 
We now proceed as in Lemma \ref{pqcomp-p} and arrive at \rif{connecting}, which is this time replaced by 
\begin{eqnarray}
\notag &&\hspace{-5mm}  \int_{B_{R}} \left(|V_p(Du)-V_p(D \tilde h )|^2+ a_{{\rm i}}(B_R) |V_q(Du)-V_q(D \tilde h )|^2\right) \, dx\\ 
\notag && \leq c\int_{B_{R}} \left[ 
F(x_B, (u)_{B_{R}}, Du)-F(x_B, (u)_{B_{R}}, D \tilde h )\right] \, dx \\
\notag && = c\int_{B_{R}} \left[ 
F(x_B, (u)_{B_{R}}, Du)-F(x, (u)_{B_{R}}, Du)\right] \, dx \\
\notag  && \quad +  c\int_{B_{R}} \left[ 
F(x, (u)_{B_{R}}, Du)-F(x, u, Du)\right] \, dx \\
\notag && \quad + c\int_{B_{R}} \left[ 
F(x, u, Du)-F(x, \tilde h, D \tilde h )\right] \, dx \\
\notag && \quad +  c\int_{B_{R}} \left[ 
F(x, \tilde h, D \tilde h )-F(x, (\tilde h)_{B_{R}}, D \tilde h )\right] \, dx \\
\notag && \quad +  c\int_{B_{R}} \left[ 
F(x, (\tilde h)_{B_{R}}, D \tilde h )-F(x, (u)_{B_{R}}, D \tilde h )\right] \, dx \\
&& \quad +  c\int_{B_{R}} \left[ 
F(x, (u)_{B_{R}}, D \tilde h )-F(x_B, (u)_{B_{R}}, D \tilde h )\right] \, dx=: c \sum_{i=1}^6 II_i\;.\label{combw}
\end{eqnarray} 
We can now proceed estimating the terms $II_1, \dots, II_6$ essentially using \rif{assF}$_{3,4}$ and \rif{applica}. 
We have
\begin{eqnarray*}II_1 + II_6 & \stackleq{ulteriore} &  
c\, \omega (2R) \int_{B_{R}}\left[H(x_B, Du)+ H(x_B, D \tilde h)\right]\, dx\\
 & &  \ +
c\, \omega (2R) \int_{B_{R}}\left[H(x, Du)+ H(x, D \tilde h)\right]\, dx\\ && \  + c [a]_{0,\alpha}R^\alpha
 \int_{B_{R}}\left(|Du|^q + |D \tilde h|^q\right) \, dx\\
 & \leq&  
c\, \omega (2R) \int_{B_{R}}\left[H(x, Du)+ H(x, D \tilde h)\right]\, dx \\ && \  + c [a]_{0,\alpha}R^\alpha
 \int_{B_{R}}\left(|Du|^q + |D \tilde h|^q\right) \, dx\\
& \stackleq{applica-real} &  
c\, \omega (2R) \int_{B_{R}}\left[H(x, Du)+ H(x, D \tilde h)\right]\, dx \\ && \  +  \frac cM \int_{B_{R}}
a(x)\left(|Du|^q + |D \tilde h|^q\right) \, dx\\
&\stackleq{finito3} &c\left[\omega(R) + \frac 1M\right]\int_{B_{R}}H(x, Du)\, dx\;.
\end{eqnarray*}
The estimates for the terms $II_2, II_4$ and $II_5$ can be now be obtained as in the case of the analogous terms 
$I_1, I_3$ and $I_4$, respectively, considered in Lemma \ref{pqcomp-p}. This yields, by also repeatedly using \rif{finito3}
$$
II_2 + II_4 + II_5 \leq c\, \omega (R^\gamma) \int_{B_{R}}H(x, Du)\, dx
$$
for a constant $c \equiv c(\datao)$. 
Finally, by the minimality of $u$ it follows that $II_3 \leq 0$. By combining the estimates found for the terms $II_1, \ldots, II_6$ with \rif{combw} and, taking \rif{thismeans} into account, we conclude with \rif{comp11-pq} and the proof is complete. 
\end{proof}
We can now prove the following $(p,q)$-phase counterpart of Lemma \ref{pcomp3p}:
\begin{lemma}\label{pqcomp2} Under the assumptions and notations of Lemma \ref{pqcomp}, for every number $\tau\in (0,1/2)$, the inequality
\eqn{seconda-p2} 
$$
\int_{B_{\tau R}}H(x, Du)\, dx \leq c_{p,q}\left\{\tau^n+ \omega(R^\gamma) + \frac1M\right\} \int_{B_{R}}H(x, Du)\, dx\;, 
$$
holds for a constant $c_{p,q} \equiv c_{p,q} (\datao)$, which is otherwise independent of $M$ and $\tau$. 
\end{lemma}
\begin{proof} The integrand $z \mapsto F(x_B, (u)_{B_{R}}, z)$ of the functional in \rif{frolast2} satisfies assumptions \rif{growtH} and \rif{assF}$_{1,2}$ with $a(x)\equiv \ai(B_{R})$. We are again in position to apply the results from \cite{Lieb}, yielding the a priori estimate
$$
\sup_{B_{R/2}}\,\left( |D \tilde h |^p+a_{{\rm i}}(B_R)|D \tilde h |^q\right) \leq c \mean{B_{R}} \left(|D \tilde h |^p+a_{{\rm i}}(B_R)|D \tilde h |^q\right)\, dx
$$
for $c$ depending only on $n, p, q, \ratio $. Thanks to \rif{thismeans}, the last inequality commutes into
\eqn{apriorisup}
$$
\sup_{B_{R/2}}\,H(x, D \tilde h) \stackleq{thismeans} \mean{B_{R}} H(x, D \tilde h)\, dx  \;.
$$
Then we proceed comparing $\tilde h$ and $u$ as follows:
\begin{eqnarray*}
&&  \int_{B_{\tau R}} H(x, Du)\, dx \leq  c\int_{B_{\tau R}} H(x, D \tilde h)\, dx \\
&& \qquad \quad +c \int_{B_{R}} \left(|V_p(Du)-V_p(D\tilde h)|^2+a_{{\rm i}}(B_R)|V_q(Du)-V_q(D\tilde h)|^2\right)\, dx\\
&& \quad \stackrel{\rif{comp11-pq}, \rif{apriorisup}}{\leq} c\tau^n \int_{B_{R}} H(x, D \tilde h)\, dx
+c\left[\omega(R^\gamma) + \frac 1M\right] \int_{B_{R}} H(x, Du) \, dx\;,
\end{eqnarray*}
where $c\equiv c (\datao)$. 
Using \rif{finito3} to estimate the first integral in the last line finally yields \rif{seconda-p2} and the proof is complete. 
\end{proof}

\section{Morrey decay, exit times and Theorem \ref{main2}}\label{morreysec}
In this section we prove Theorem \ref{main2}. The proof employes \rif{convex-omega-2} but not \rif{convex-omega}, 
and follows combining Lemmas \ref{pcomp3p} and \ref{pqcomp2} from the last two sections. We shall use an exit time argument, that, in a different form, has been introduced in \cite{CM1, CM2} to treat the regularity of minimizers of the model functional $\mathcal P$ in \rif{model}. The final outcome is the Morrey type decay estimate \rif{morrey}. This eventually implies the local H\"older continuity of $u$ for every exponent $\theta<1$ claimed in Theorem \ref{main2} via the usual integral characterization of Campanato and Meyers (see for instance \cite[Chapter 2]{G}) and a standard covering argument. We start fixing a number $\sigma \in (0,n)$; moreover, for every ball $B_{R}\equiv B_{R}(x_0)\Subset \Omega$ we abbreviate as follows:
$$
\mathcal H(\varrho):= \int_{B_{\varrho}(x_0)} H(x,Du)\, dx \qquad \forall \ \varrho \leq  R\;.
$$
\indent {\em Step 1: Iteration in the $(p,q)$-phase.} Let us consider a ball $B_{R}\equiv B_{R}(x_0)\Subset \Omega_0\Subset \Omega$, $R\leq 1$ and assume that \rif{applica} holds for some $M\geq 1$ to be chosen in a few lines. We start recalling the basic property in \rif{ognir}. 
This, in perspective, allows to iteratively apply Lemma \ref{pqcomp2} on all smaller scales once it can be applied on a fixed, initial scale.
When using Lemma \ref{pqcomp2}, inequality \rif{seconda-p2} can be rewritten as
\eqn{HH1}
$$
\mathcal H(\tau R) \leq \tau^{n-\sigma}\left\{c_{p,q} \tau^\sigma + c_{p,q}  \tau^{\sigma-n}\omega(R^\gamma)+ \frac{ c_{p,q} \tau^{\sigma-n}}M\right\} \mathcal H( R)\;,
$$
where the constant $c_{p,q} $ has been determined in Lemma \ref{pqcomp2} as a function of $\datao$ but not of $M$. We start taking 
$\tau \equiv  \tau_1$ small enough in order to get 
\eqn{scelta1}
$$ \tau_1^\sigma\leq \frac1{3c_{p,q}} \Longrightarrow \tau_1\equiv \tau_1(\datao, \sigma)\;.$$ With this choice we determine a (potentially small) radius $R_1\equiv R_1(\datao, \sigma)\leq 1$ and a (potentially large) number $M\equiv M(\datao, \sigma) \geq 1 $, in order to satisfy the two inequalities 
\eqn{scelta2}
$$
 \omega(R_1^\gamma) \leq \frac {\tau^{n-\sigma}_1}{3c_{p,q} } \qquad \mbox{and}\qquad \frac{ 1}M \leq 
\frac {\tau_1^{n-\sigma}}{3c_{p,q} }\;.
$$
With such choices \rif{HH1} implies
\eqn{HH2}
$$
\mathcal H(\tau_1 R) \leq \tau_1^{n-\sigma}\mathcal H( R)\;, 
$$
for every $R$ such that $R \leq R_1$. 
By using \rif{ognir} we deduce that since \rif{applica} holds, then we also have $a_{{\rm i}}(B_{\tau^k R}) > M[a]_{0,\alpha}(\tau^kR)^{\alpha-s}$ holds for every integer $k \geq 0$. We can therefore iterate \rif{HH2}, thereby obtaining 
\eqn{HH3}
$$
\mathcal H(\tau_1^k R) \leq \tau_1^{(n-\sigma)k}\mathcal H( R) \qquad  \mbox{for every integer} \ k \geq 0\;. 
$$
\indent {\em Step 2: Decay in the $p$-phase.} The choices in \rif{scelta1} and \rif{scelta2} fix, in particular, the value of the threshold constant $M$ as a function of $\datao$ and $\sigma$. This is the value we are going to use here when applying Lemma \ref{pcomp3p}; in particular, this choice determines $c_p$ and $s_0$ as functions of $\datao$ and $\sigma$ only. Moreover, we now take a fixed choice of the exponent $s$ in \trif{applica-pt}, say 
$s :=s_m/2$. With $M$ and $s$ having been fixed, the radius $R_*$ from \rif{raggiostella0} is now a function of $\datao, \sigma$ and $\bar \eps$. 
Again, let us consider, as in Step 1, a ball $B_{R}\Subset \Omega_0\Subset \Omega$. This will be here considered with a radius $R$ that will initially be such that $R \leq R_1$ and $R_1\equiv R_1(\datao, \sigma)\leq 1$ has been determined in the previous step; further restrictions will be taken in a short while. In order to apply 
Lemma \ref{pcomp3p} with the mentioned choice of $M$, we look at estimate \rif{seconda-p1}, that reads as 
\eqn{seconda-pp} 
$$
\mathcal H (\tau R) \leq \tau^{n-\sigma}\left\{c_p\tau^\sigma+ c_pc_\eps \tau^{\sigma-n-p} \bar \eps + c_p\tau^{\sigma-n-p}\eps\right\} \mathcal H (R)\;, 
$$
for every choice of $\bar \eps, \eps \in (0,1) $, provided we are choosing $R \leq R_*(\datao, \sigma,  \bar \eps)$. 
We choose $\tau\equiv \tau_2$ such that
\eqn{scelta3}
$$ \tau_2^\sigma\leq \frac1{3c_{p}} \Longrightarrow \tau_2\equiv \tau_2(\datao, \sigma)\;.$$
We then proceed with the choice of $\eps\equiv \eps(\datao, \sigma)$ such that 
\eqn{scelta33} 
$$
\eps \leq \frac{\tau^{n+p-\sigma}}{3c_p}\;.
$$
This, in turn, determines the constant $c_{\eps}$ in \rif{seconda-pp} as a function of $\datao$ and $\sigma$. We then take $\tilde \eps\equiv \bar \eps(\datao, \sigma) $ small enough to guarantee that
\eqn{scelta4} 
$$
 \bar \eps \leq \frac{\tau^{n+p-\sigma}}{3c_pc_\eps}\;.
$$
The choice in \rif{scelta4}, via \rif{raggiostella0}, in turn determines a radius 
\eqn{raggiofinale}
$$R_2:=R_*(\datao, \sigma)\leq R_1(\datao, \sigma)\leq 1$$
such that \rif{seconda-pp} is satisfied according to the prescription of Lemma \ref{pcomp3p}. 

Summarizing, with the choices of $\tau_2$ and $R_2$ displayed in \rif{scelta3}-\rif{scelta4} and used in \rif{seconda-pp}, we can conclude that if $R\leq R_2\leq R_1$ and if the condition 
\eqn{iteracond} 
$$
a_{{\rm i}}(B_{\tau_2 R}) \leq M[a]_{0,\alpha}(\tau_2 R)^\alpha\;,
$$ 
is satisfied, then we have
\eqn{HH4}
$$
\mathcal H(\tau_2 R) \leq \tau^{n-\sigma}_2\mathcal H(R)\;. 
$$

{\em Step 3: Exit time and iteration.} We combine \rif{HH3} and \rif{HH4} via the announced exit time argument and take a general ball $B_{R}\equiv B_{R}(x_0)\Subset \Omega_0\Subset \Omega$ with $R\leq R_2$ and $R_2$ as in \rif{raggiofinale}. By looking at \rif{iteracond}, we consider conditions of the type
\eqn{iteracond2} 
$$
a_{{\rm i}}\left(B_{\tau^{k+1}_2 R}\right) \leq M[a]_{0,\alpha}\left(\tau^{k+1}_2 R\right)^\alpha\;, 
$$ 
for every integer $k \geq 0$. 
We define the exit time index
\eqn{exit-index}
$$
m :=\min\{k\in \en \, \colon \, \mbox{condition \rif{iteracond2} fails}\}
$$
(notice that it can happen that $m=0$). By the content of Step 2, and in particular applying \rif{HH4} repeatedly, we have
\eqn{itera4} 
$$
\mathcal H(\tau_2^k R) \leq \tau_2^{(n-\sigma)k}\mathcal H(R)\qquad \forall \ k \in \{0, \ldots, m\}\;.
$$
Condition \rif{iteracond2} fails for $k=m$ and this means
\eqn{occurrenceof}
$$
a_{{\rm i}}\left(B_{\tau_2^{m+1}R}\right) > M[a]_{0,\alpha}(\tau_2^{m+1}R)^\alpha\;.$$ The occurrence of \rif{occurrenceof} allows to apply the argument of Step 1, and in particular allows to use \rif{HH3} (with $B_{R}$ now replaced by $B_{\tau_2^{m+1}R}$); we therefore conclude that
\eqn{itera5} 
$$
\mathcal H(\tau_1^{k}\tau_2^{m+1} R) \leq \tau_1^{(n-\sigma)k}\mathcal H(\tau_2^{m+1}R)\qquad \mbox{for every integer} \ k \geq 0\;.
$$  
\indent {\em Step 4: Conclusion.} We are ready to establish \rif{morrey}. For this we first assume that $0 < \varrho < R\leq R_2$ and $R_2$ is in \rif{raggiofinale}. We then distinguish various cases. The first occurs when $\tau_2^{m+1}R \leq \varrho$ and $m$ is in \rif{exit-index}; in this situation we find $\bar k \in \{0, \ldots, m\}$ such that $\tau_2^{\bar k+1}R \leq \varrho \leq \tau_2^{\bar k}R$ and therefore $\tau_2^{\bar k+1} \leq \varrho/R$. By using this last relation we can estimate
\begin{eqnarray}
\notag \mathcal H(\varrho) \leq \mathcal H(\tau_2^{\bar k}R) &\stackleq{itera4} & \tau_2^{(n-\sigma)\bar k}\mathcal H( R) \\ &\leq & \tau_2^{\sigma-n}
(\varrho/R)^{n-\sigma}\mathcal H( R) \stackrel{\rif{scelta3}}{\equiv} c(\datao, \sigma) (\varrho/R)^{n-\sigma}\mathcal H( R) \label{hh3}\;,
\end{eqnarray}
where $c \equiv c (\datao, \sigma)$; therefore \rif{morrey} follows in this case provided $R\leq R_2$. We then consider the occurrence of $\varrho < \tau_2^{m+1}R$, and make a further distinction. We first treat the case 
when $\tau_1\tau_2^{m+1}R \leq \varrho $. In this situation we estimate
\begin{eqnarray*}
\notag
\mathcal H(\varrho) \leq \mathcal H(\tau_2^{m+1}R)  &\stackleq{hh3} & c\tau_2^{(n-\sigma)(m+1)}\mathcal H( R)
\\ &\leq & c\tau_1^{\sigma-n}
(\varrho/R)^{n-\sigma}\mathcal H( R) \\ &\stackrel{\rif{scelta1}}{\equiv}& c(\datao, \sigma) (\varrho/R)^{n-\sigma}\mathcal H( R)\;.
\end{eqnarray*}
Finally, in the case $\varrho < \tau_1\tau_2^{m+1}R$ we again find $\tilde k\geq 1$ such that $\tau_1^{\tilde k+1}\tau_2^{m+1}R \leq \varrho \leq \tau_1^{\tilde k}\tau_2^{m+1}R$ and therefore we have
\begin{eqnarray*}
\notag
\mathcal H(\varrho) \leq \mathcal H(\tau_1^{\tilde k}\tau_2^{m+1}R)  &\stackleq{itera5} & \tau_1^{(n-\sigma)\tilde k}\mathcal H(\tau_2^{m+1}R)
\\ &\stackleq{hh3} & c\left(\tau_1^{\tilde k}\tau_2^{m+1}\right)^{n-\sigma}\mathcal H(R) \\
&\leq &  c\tau_1^{\sigma-n}
(\varrho/R)^{n-\sigma}\mathcal H( R) \equiv c (\varrho/R)^{n-\sigma}\mathcal H( R)\;,  
\end{eqnarray*}
where again it is $c\equiv c(\datao, \sigma)$ again by the dependence on the constants in \rif{scelta1}. 
Summarizing the estimates in the last three displays, we have established 
\eqn{hh1000}
$$
\mathcal H(\varrho) \leq c(\datao, \sigma) (\varrho/R)^{n-\sigma} \mathcal H(R) \qquad 0 < \varrho \leq R \leq R_2\;, 
$$
so that \rif{morrey} follows in this case too provided $R \leq R_2$, $R_2$ is as in \rif{raggiofinale}. We finally settle the remaining  case $R_2 < R \leq 1$. For $R_2 \leq \varrho \leq R\leq 1$, we can trivially estimate
\begin{eqnarray*}\mathcal H(B_{\varrho}) &\leq & (\varrho/R)^{\sigma-n} (\varrho/R)^{n-\sigma}\mathcal H(B_{R})\\ & \leq &R_2^{\sigma-n}(\varrho/R)^{n-\sigma}\mathcal H(B_{R})\stackrel{\rif{raggiofinale}}{\equiv} c(\datao, \sigma)(\varrho/R)^{n-\sigma}\mathcal H(B_{R})\;.
\end{eqnarray*} On the other hand, if 
$\varrho \leq R_2 \leq R\leq 1$, we similarly have
$$\mathcal H(B_{\varrho})
 \stackleq{hh1000} c (\varrho/R_2)^{n-\sigma}\mathcal H(B_{R_2})
   \leq  R_2^{\sigma-n}(\varrho/R)^{n-\sigma}\mathcal H(B_{R})
   \equiv c(\varrho/R)^{n-\sigma}\mathcal H(B_{R})\;.$$ 
We have therefore established the decay estimate in \rif{morrey} for every possible choice of $\varrho, R$ such that 
$0<\varrho \leq R\leq 1$ and the proof of Theorem \ref{main2} is complete. 

\section{Gradient continuity and Theorem \ref{main1}}\label{gradsec}
In this section we prove Theorem \ref{main1}. The way we combine the $p$- and $(p,q)$-phases is different. Specifically, we shall use condition \rif{applica}, again with $s=0$, but with $M$ depending on $R$, in a way that makes it blow-up when $R \to 0$. This would create a bad dependence on the constants when using Lemma \ref{pcomp} with the same choice of $s$. To eliminate this bad dependence we then consider the $p$-phase \rif{applica-p} with a suitable positive exponent $s$, depending on the way $M$ has been chosen. Iterating this argument along with a comparison scheme finally leads to desired H\"older continuity. To proceed, we first appeal to Theorem \ref{main2}. This, via a standard covering argument, gives that for every open subset $\Omega_0 \Subset \Omega$ and $\kappa>0$, there exists a constant $c\equiv c(\datao, \kappa)$ such that 
\eqn{morrey1}
$$
\mean{B_r} H(x, Du) \, dx \leq c r^{-\kappa}
$$
holds for every ball $B_r\Subset \Omega_0 \subset \Omega$. 
We then consider a ball $B_{R} \Subset \Omega_0$, $R\leq 1$, and assume that condition \rif{applica} holds with the choice $s=0$ and $M\equiv M(R)$ as follows:
\eqn{applica-fin}
$$
a_{{\rm i}}(B_{R}) > M(R)[a]_{0,\alpha}R^{\alpha} \quad \mbox{with} \quad M \equiv M(R):=R^{-s_{m}/2}\;,$$
where $s_m$ has been defined in \rif{esseemme}. Using \rif{comp11-pq} and recalling \rif{convex-omega}, we conclude that the inequality
\begin{eqnarray}
\notag &&\mean{B_{R}} \left(|V_p(Du)-V_p(D \tilde h )|^2+a_{{\rm i}}(B_{R})|V_q(Du)-V_q(D \tilde h )|^2\right)\, dx\\ && \quad  \leq cR^{\min\{\gamma\beta, s_m/2\}} \mean{B_{R}} H(x, Du) \, dx\;,\label{diffa1}
\end{eqnarray}
holds for a constant $c\equiv c(\datao)$. 
Here $\tilde h$ is defined as in \rif{var-dir1}.  
Then, we consider the case \rif{applica-fin} does not hold, and this leads to
\eqn{applica-fin2}
$$
a_{{\rm i}}(B_{R}) \leq  M(R)[a]_{0,\alpha}R^{\alpha}= [a]_{0,\alpha}R^{\alpha-s_m/2} \;.$$
This means that condition \rif{applica-p} holds with the choice $M=1$ and $s=s_m/2$. We are therefore in position to apply Lemma \ref{pcomp-dopo}, the advantage being now that, since now $M=1$, then the constant $c$ and the exponent $s_1$ appearing in \rif{comp11dd} are determined as functions $\datao$ while no dependence with respect to $R$ occurs. Inequality \rif{comp11dd} now gives
\begin{eqnarray}
\notag &&\mean{B_{R/8}} \left(|V_p(Du)-V_p(D \tilde h )|^2+a_{{\rm i}}(B_{R})|V_q(Du)-V_q(D \tilde h )|^2\right)\, dx\\ && \quad  \leq cR^{\min\{\gamma \beta s_1, s_1s_m/(2q)\}} \mean{B_{R}} H(x, Du) \, dx\;.\label{diffa2}
\end{eqnarray}
Summarizing the contents of \rif{diffa1} and \rif{diffa2} and combining them with \rif{morrey1} for the choice 
\eqn{ilkappa}
$$\kappa\equiv \kappa_1\equiv \kappa_1(\datao):=\frac 12\min\left\{\gamma\beta, \gamma\beta s_1,\frac{s_m}{2}, \frac{s_1s_m}{2q}\right\}\;,$$
we conclude with 
\eqn{diffVV}
$$
\mean{B_{R/8}} \left(|V_p(Du)-V_p(D \tilde h )|^2+a_{{\rm i}}(B_{R})|V_q(Du)-V_q(D \tilde h )|^2\right)\, dx 
\leq c R^{\kappa_1}\;,
$$
and $c\equiv c (\datao)$. 
Notice that here $\tilde h$ is defined either via \rif{var-dir2} or via Lemma \ref{pcomp-dopo} (and minimizes 
the functional in \rif{frolast22}), accordingly to which inequality \rif{applica-fin} or 
\rif{applica-fin2} comes into the play, respectively. Therefore it satisfies
\eqn{finito2222}
$$ \mean{B_{R/8}} \left(|D \tilde h|^p+a_{{\rm i}}(B_{R})|D \tilde h|^q\right) \, dx \leq c\mean{B_{R}} H(x, Du) \, dx\;,$$
for $c\equiv c (n, \ratio)$. This follows from estimates \rif{finito2} and \rif{stima-energia}.  
We can conclude with a comparison scheme, which follows with some variants the one in \cite{CM1}. Notice that, arguing exactly as in \rif{comequi2}-\rif{comequi3} (take $\theta\equiv 1$ and $a_0\equiv a_{{\rm i}}(B_{R})$ there), from the relation \rif{diffVV} it follows that
\eqn{final-diff}
$$
\mean{B_{R/8}} \left(|Du-D \tilde h |^p+a_{{\rm i}}(B_{R})|Du-D \tilde h |^q\right)\, dx \leq c R^{\kappa_2}\;,  
$$
for a new positive exponent $\kappa_2\equiv \kappa_2 (\kappa_1, n,p,q)\in (0,1)$, computable as in \rif{comequi2}-\rif{comequi3}
once $\kappa_1$ is fixed as in \rif{ilkappa}; all in all we have 
$\kappa_2\equiv \kappa_2 (\datao)$. We observe that in both cases \rif{var-dir1} and \rif{var-dir2}, $\tilde h$ is a minimizer of a functional whose integrand depends only on the gradient and satisfies assumptions \rif{assF}$_{1,2}$ with $a(x)\equiv \ai(B_{R})$. The theory in \cite{Lieb} applies and provides the following gradient H\"older decay estimate, valid whenever $0< \varrho \leq R/8$:
\begin{eqnarray}
\notag && \mean{B_\varrho} \left( |D\tilde h-(D\tilde h)_{B_\varrho}|^p+a_{{\rm i}}(B_{R})
|D\tilde h-(D\tilde h)_{B_\varrho}|^q\right)\, dx\\
&& \quad \leq c \left(\frac{\varrho}{R}\right)^{p\tilde \beta} \mean{B_{R}} (|D\tilde h|^p+a_{{\rm i}}(B_{R})|D\tilde h|^q)\, dx
\stackleq{finito2222} c \left(\frac{\varrho}{R}\right)^{p\tilde \beta} \mean{B_{R}} H(x, Du)\, dx\,, \label{Ldecay}
\end{eqnarray}
and where $c$ and $\tilde \beta$ both depend only on $n,p,q,\ratio$, but are otherwise independent of the specific value of 
$a_{{\rm i}}(B_{R})$; with no loss of generality we assume it is $p\tilde \beta<1$. Then, by also using \rif{elem-prop}, for $0< \varrho \leq R/8$ we have
\begin{eqnarray}
 \notag && \mean{B_\varrho} |Du-(Du)_{B_\varrho}|^p \, dx \leq   c
 \mean{B_\varrho} |D \tilde h-(D \tilde h)_{B_\varrho}|^p \, dx + c\mean{B_\varrho} |Du-D \tilde h|^p \, dx\\
   && \qquad \stackleq{Ldecay}  c
  \left(\frac{\varrho}{R}\right)^{p\tilde \beta}\mean{B_{R}} H(x, Du) \, dx + c\left(\frac{R}{\varrho}\right)^{n}
  \mean{B_{R/8}} |Du-D \tilde h|^p\, dx\notag \\ 
     && \ \ \ \stackrel{\rif{morrey1}, \rif{final-diff}}{\leq}  c
  \left(\frac{\varrho}{R}\right)^{p\tilde \beta}R^{-\kappa} + c\left(\frac{R}{\varrho}\right)^{n}
  R^{\kappa_2}
  \label{morreyqq}\end{eqnarray}
for $c\equiv c(\datao, \kappa)$; notice that $\kappa\in (0,1)$ can still be chosen, while $\kappa_2$ has been fixed through \rif{final-diff} and depends on $\datao$. Choosing 
$\kappa\equiv \kappa_2p\tilde \beta/(8n)$, and taking $\varrho= (R/8)^{1+\kappa_2/(4n)}$ in \rif{morreyqq}, after a few elementary manipulations we get that 
\eqn{finalmente}
$$
\mean{B_\varrho} |Du-(Du)_{B_\varrho}|^p \, dx \leq c \varrho^{\frac{\kappa_2p\tilde \beta}{16n}}$$
holds for every $\varrho \in (0,1/8)$, provided $B_{8\varrho} \Subset \Omega_0$. 
By the already invoked integral characterisation of H\"older continuity due to Campanato and Meyers, and a standard covering argument, 
\rif{finalmente} implies that $Du \in 
C^{0,\beta_0}_{\rm{loc}}(\Omega)$ for $\beta_0=\kappa_2\tilde\beta/(16n)$. This proves the local H\"older continuity of $Du$ but not yet the full statement of Theorem \ref{main1}, that claims that the H\"older continuity exponent $\beta_0$ of $Du$ depends only on $n,p,q,\ratio, \alpha$ and $\beta$, while we have that the exponent found $\kappa_2\tilde\beta/(16n)$ 
depends on $\datao$. This exponent can be upgraded to reach the required dependence. 
Indeed, once we know that the gradient is locally bounded, 
the non-uniform ellipticity of the functional $\mathcal F$ becomes immaterial and we can apply some  
standard perturbation methods in order to obtain full statement of Theorem \ref{main1}. We briefly summarize the steps. We go back to Section \ref{fasepq} and perform the same procedure 
as for Lemma \ref{pqcomp}; we do this in any case, independently of the 
occurrence of condition \rif{applica}. Let us observe that the functional in \rif{frolast22} satisfies the Bounded Slope Condition (see for instance \cite{brasco}) and therefore there exists a constant $c\equiv c (n,p,q,\ratio, \|Du\|_{L^{\infty}(B_R)})$ such that 
$\|D\tilde h\|_{L^{\infty}(B_R)}\leq c$. Using this fact, the estimation of the terms $II_1, \ldots, II_6$ in Lemma \ref{pqcomp}
simplifies and leads to replace \rif{comp11-pq} by the stronger
\eqn{comp11-pq-fine0}
$$\mean{B_{R}} |V_p(Du)-V_p(D \tilde h )|^2\, dx  \leq cR^{\min\{\beta, \alpha\}} \;, 
$$
where the constant now depends on $n,p,q,\ratio, \|Du\|_{L^{\infty}(\Omega_0)}, \|a\|_{L^{\infty}(\Omega_0)}$. Notice also that we can 
take $\gamma=1$ in Lemma \ref{pqcomp} as now $u$ is locally Lipschitz. Eventually, \rif{comp11-pq-fine0} gives
\eqn{comp11-pq-fine}
$$\mean{B_{R}} |Du-D \tilde h |^2\, dx  \leq cR^{\kappa} \;, 
$$
where $\kappa = \min\{\beta, \alpha\}$ if $p\geq 2$ and $2\kappa = p\min\{\beta, \alpha\}$ otherwise.  Similarly, \rif{Ldecay} now simplifies in 
\eqn{Ldecay-fine}
$$\mean{B_\varrho}  |D\tilde h-(D\tilde h)_{B_\varrho}|^p\, dx \leq c  \left(\frac{\varrho}{R}\right)^{p\tilde \beta} $$
where $\tilde \beta$ depends only on $n,p,q, \ratio$, 
and the constant $c$ depends also on $\|Du\|_{L^{\infty}(\Omega_0)}$ and 
$\|a\|_{L^{\infty}(\Omega_0)}$. The inequalities in \rif{comp11-pq-fine}-\rif{Ldecay-fine} can be now combined in the 
same way \rif{morreyqq} and \rif{finalmente} have been combined above to deduce the gradient H\"older continuity. This time all the exponents involved depend only on 
$n,p,q,\ratio, \alpha, \beta$. This means that the local H\"older continuity of $Du$ follows with an exponent $\beta_0$
with the dependence on the various constants described in the statement and the proof is finally complete. 
 
\section{Interpolative effects and Theorem \ref{main3}}
The proof of Theorem \ref{main3} upgrades the one for Theorem \ref{main1}. 
In the following, the definition in \rif{esseemme} will be replaced by $
s_m:=p+\alpha-q+\gamma(q-p),
$
where $\gamma \in (0,1]$ is the H\"older continuity exponent coming from assumptions \rif{bound3} and not 
any longer the one coming from Theorem \ref{HO}, whose assumptions are in general not implied by  
\rif{bound3}. In this respect, we shall restrict to the case $p\leq n/(1-\gamma)$. This is the range for which \rif{bound3} is a weaker assumption than \rif{bound1}. See also the comments in the Introduction after Theorem \ref{main3}. 
\subsection{Step 1: Blow-up, first time.} We start revisiting the proofs in Section \ref{fasep} in order get Lemma \ref{pcomp3p} in this case, with obviously a different function $R_*$ defined in \trif{raggiostella0}. We notice that since the functional of Theorem \ref{main3} does not depend on the $u$-variable, then there is no reason to go through Lemma \ref{pqcomp-p} and in the following we shall revise the proofs of Section \ref{fasep} with $\tilde v \equiv u$. For this reason from Lemma \ref{auxv} we only retain \rif{aux-sup}. We then go thorough Lemma \ref{pcomp}, where $v$ defined in \rif{defiv}, is actually the blow-up of the original minimizer $u$ in $B_R\equiv B_R(x_0)$:
\eqn{newdefv}
$$v(x) := \frac{ u(x_0+Rx)-(u)_{B_{R/2}}}{ E(R) }\,, \quad  E(R) := \left(\mean{B_{R/2}} |u - \left(u \right)_{B_{R/2}} |^p \, dx \right)^{1/p}\;.
$$
The main point is that we shall use the non-quantitative harmonic approximation of Lemma \ref{p-harm2} instead of the quantitative one in Lemma \ref{p-harm}. This essentially depends on the fact that at this stage we are not yet able to verify the conditions to apply Theorem \ref{HI} to $v$ and get the higher integrability in \rif{nuova-maggiore}, which is essential for Lemma \ref{p-harm}. For the same reason, we can by-pass Step 2 of the proof of Lemma \ref{pcomp}. We therefore go to Step 3 for Lemma \ref{pcomp} and note that 
the computations for in \rif{stimaI2}-\rif{stimaI3} keep on being true in this setting; see also Remark \ref{stupido2}. 
These lead to establish \rif{appi2} also in this case.  
Now, fix $\tilde \eps \in (0,1)$ and determine the corresponding $\sigma$ from Lemma \ref{p-harm2}. We can find 
\eqn{nuovoraggio}
$$R^*\equiv R^*(\datao, \omega(\cdot), M, \tilde \eps)>0$$ such that $c_h  [\texttt{o}(R)] \leq \sigma$ whenever $R \leq R^*$, so that \rif{appi2} gives that
$$
\left|\mean{B_{1/8}} \big\langle \partial \bar F_0 (D v), D\varphi  \big\rangle\, dx \right|\leq \sigma
\|D\varphi\|_{L^{\infty}}$$
 holds for every function $\varphi \in W^{1, \infty}_0(B_{1/8})$. Recall that in the present setting $\bar F_0(\cdot)$ is defined as $z \mapsto [R/E(R)]^{p}F\left(x_B,  [E(R)/R]z\right)$. We are therefore in condition 
 to apply Lemma \ref{p-harm2}. This yields the existence of a function $h \in u+ W^{1,\bar H_0}_0(B_{1/8})$, with $\bar H_0(\cdot)$ as in \rif{ultimaH}, 
 such that \rif{p-har-dopo}
 holds for all $\varphi \in W^{1,\infty}_0(B_{1/8})$ and such that
\eqn{stimaHH2}
$$
\mean{B_{1/8}} \bar H_0(Dh)\, dx \leq c(\datao, M)
$$
is true and \trif{stimadiff-final} holds. These are exactly the ingredients to proceed as in Step 6, thereby getting the validity of Lemma \ref{pcomp}. We can then proceed as in Lemma \ref{pcomp3p} - 
actually in an easier way as $\tilde v \equiv u$. Indeed, \rif{lasty} does not any longer take place and \rif{stimadec0} (recall $\tilde v\equiv u$) directly implies
$$
\int_{B_{\tau R}} \left|\frac{u-(u)_{B_{\tau R}}}{\tau R}\right|^p \, dx  \leq c\left\{\tau^{n} + \tau^{-p}\bar \eps \right\}\int_{B_{R}}
H(x, Du)\, dx\;, 
$$
for every $\bar \eps\in (0,1)$, $\tau \in (0, 1/32)$, where $c\equiv c (\datao, M)$ and provided $R \leq R^*$, with $R^*$ being defined in \rif{nuovoraggio}. 
Arguing exactly as after \rif{stimadec0-dopo} we get the assertion of Lemma \ref{pcomp3p}, with this time $R_*\equiv R^*$. This concludes the analysis of the $p$-phase. It is now not very difficult to see that all the arguments of Sections \ref{fasepq}  
and \ref{morreysec} can be reproduced verbatim, thereby leading to the analog of Theorem \ref{main2} under 
assumptions \rif{bound3}, as announced in the statement of Theorem \ref{main3}. In particular, estimate \rif{morrey} holds and this will be exploited in the next step.

\subsection{Step 2: $C^{0, \gamma}$-regularity in rigid form.} We now consider a ball $B_R \Subset \Omega_0\Subset \Omega$ 
with $R\leq 1$ and fix $\theta \in (0,1)$. In estimate \rif{morrey} we take $\sigma =p-p\theta$. Using Poincar\'e inequality in a standard way we get that 
 \eqn{morrey-dopo}
$$
\mean{B_{\varrho}} |u-(u)_{B_\varrho}|^p\, dx \leq 
c\left[r^{p-p\theta}\mean{B_{r}} H(x,Du) \, dx\right] \varrho^{p\theta}
$$
holds whenever $B_{\varrho}\subset B_r$ are concentric balls that are contained in $B_R$ (but this time not necessarily concentric to $B_R$), and 
where $c\equiv c (\datao, \theta)$. We are now able to apply the standard integral characterization of H\"older continuity due to Campanato and Meyers, that gives, together with a standard covering argument yields 
the following local $C^{0, \theta}$-estimate for $u$: 
\eqn{localap}
$$
R^{\theta}[u]_{0, \theta;B_{R/4}} \leq c \left[R^{p}\mean{B_{3R/8}} H(x,Du) \, dx\right]^{1/p}\;.
$$
The constant $c$ depends on $\datao$. We now further assume that condition \rif{applica-p} holds with $s=0$ 
and for some $M\geq 1$. 
This immediately gives the validity of inequality \rif{cacc1} from Lemma \ref{cacc-gen2}; 
see also Remark \ref{stupido00}. Taking $\theta=\gamma$ in \rif{localap}, and using \rif{cacc1} 
- with $t=3R/8$ and $s=R/2$ - we conclude with
\eqn{stimablow}
$$
R^{\gamma}[u]_{0, \gamma;B_{R/4}} \leq c \left(\mean{B_{R/2}} |u-(u)_{B_{R/2}}|^p \, dx\right)^{1/p}= cE(R)\;, 
$$
where $c\equiv c(\datao, M)$. 
This is the estimate we were looking for. Notice that this estimate, from a qualitative viewpoint, tells no new 
about the regularity of $u$, which is indeed assumed to be $C^{0,\gamma}$-regular. 
The main point here is the specific 
form of the a priori estimate \rif{stimablow}, that will eventually allow to implement a blow-up procedure totally similar to the one of Lemma \ref{pcomp}. 
\subsection{Step 3: Blow-up, second time.} The only missing part for the proof of Theorem \ref{main1} is the one of Section \ref{gradsec}, which 
is in turn based on Lemma \ref{pcomp-dopo}, with the crucial quantitative estimate \rif{comp11dd} from Lemma \ref{pcomp}.  
This, in turn, 
comes from estimate \rif{stimadiff}, which is a direct consequence of the application of Lemma \ref{p-harm}. To apply Lemma \ref{p-harm} in the setting of 
Lemma \ref{pcomp}, we need to verify the higher integrability information \rif{nuova-maggiore}. This was the missing information in Step 1, forcing us to apply the weaker Lemma \ref{p-harm2}. Estimate \rif{nuova-maggiore} comes from 
the application of Theorem \ref{HI} to $v$, that minimizes the rescaled functional in \rif{rescaleF}; this satisfies the new growth assumptions in \rif{nuovaH}-\rif{assFdd}. Under the new assumptions \rif{bound3}, the application of Theorem \ref{HI} requires essentially two ingredients: the $\alpha$-H\"older continuity of the rescaled coefficient $[E(R)/R]^{q-p}a_R(\cdot)$ and the fact that $[v]_{0, \gamma;B}$ is bounded; both information must be uniform with respect to $R$. The former can be checked exactly as in \rif{evident}, see Remark \ref{stupido}. For the latter estimate \rif{stimablow} becomes crucial. Indeed, we have 
$$
[v]_{0, \gamma;B_{1/4}} \stackrel{\rif{newdefv}}{=} \frac{R^{\gamma}[u]_{0, \gamma;B_{R/4}}}{E(R)}
 \stackleq{stimablow} c(\datao, M)\;.
$$
Summarizing, we have that $v$ is a local minimizer of the functional in \rif{rescaleF} to which we can apply Theorem 
\ref{HI}, with all the constants involved being independent of $R$.  
We get \rif{nuova-maggiore} and \rif{appi1}, with $\delta_1$ depending only on $\datao$ and $M$. This - see also Remark \ref{stupido2} - eventually allows to get \rif{p-har-dopo}-\rif{stimadiff} with constants $c$ being completely independent of $R$.  
This is sufficient to reprove Lemma \ref{pcomp-dopo} with the same dependence of the constants of the original statement, 
and in particular allows to get the comparison estimate \rif{comp11dd}. These are the essential ingredients to run the 
proof in Section \ref{gradsec} and eventually leads to the assertion of Theorem \ref{main3}. 
The rest of the proof now goes exactly as in Lemma \ref{pcomp} - see also Remark \ref{stupido2} - and in the case of Theorem \ref{main1}.

\section{Lavrentiev phenomenon and Theorem \ref{main4}}
In this section we prove Theorem \ref{main4}. The proof is unified for the two cases \rif{bound2} and \rif{bound33}. In the first one we shall formally take the parameter $\gamma$ considered in \rif{bound33} as $\gamma=0$. We start fixing a ball $B\Subset \Omega$ as in the statement of Theorem \ref{main4}. We notice that we can reduce ourselves to prove that there exists a sequence$\{u_k\} $ of $W^{1,\infty}(B)$-regular functions such that $u_k \to u$ strongly in 
 $W^{1,p}(B)$ and such that 
 \eqn{convergenza2}
 $$
 \lim_{k} \, \int_{B} H(x, Du_k)\, dx  =\int_{B} H(x, Du)\, dx\;.
 $$
In fact, once \rif{convergenza2} is at hand, then \rif{convergenza1} follows, up to a not relabelled subsequence, by a standard variant of Lebesgue's dominated convergence theorem since \rif{growtH} is in force. For $\eps \in (0, 1/2)$ sufficiently small to have $\eps \leq {\rm  dist}(B, \partial \Omega)/20$, we consider the mollified functions  
$u_\eps :=u*\rho_\eps $. Here $\{\rho_\eps\}$ is a family of standard mollifiers generated by a smooth non-negative and radial function 
$\rho \in C^\infty_0(B_1)$ such that $\|\rho\|_{L^1(\ern)}=1$, via 
$\rho_{\eps}(x) = \eps^{-n} \rho(x/\eps)$. Next, following the notation introduced in Section \ref{lefasi}, for every $x \in B$ we consider
\eqn{defih0}
$$
\ai(B_{2\eps}(x)):=\min_{y \in \overline{B_{2\eps}(x)}}a(y) \quad \mbox{and}\quad H_\eps(x,z):= |z|^p+ a_{{\rm i}}(B_{2\eps}(x))|z|^q\;.$$
We now again distinguish two different phases as done in Section \ref{lefasi}. 

{\em $p$-phase.} Here we assume that
\eqn{pfase} 
$$
a_{{\rm i}}(B_{2\eps}(x)) \leq  2[a]_{0, \alpha}\eps^\alpha\;.
$$
As seen in Remark \ref{propaga}, this also implies that 
\eqn{supfi}
$$
\|a\|_{L^{\infty}(B_{2\eps}(x))} \leq 6[a]_{0, \alpha}\eps^\alpha\;.
$$
Moreover, in the case $q \leq p+\alpha$ and we are considering \rif{bound2}, the Caccioppoli inequality \rif{cacc1} from Lemma \ref{cacc-gen2} applies by \rif{pfase} and gives, in particular
\eqn{cacc-lav}
$$
\mean{B_{\eps}(x)} |Du|^p\, dy \leq c\mean{B_{2\eps}(x)}\left|\frac{u-(u)_{B_{2\eps}(x)}}{\eps}\right|^p\, dy\;.
$$
This inequality continues to hold when assuming \rif{bound33}. Indeed, in any case, by \rif{cacc0} we have
$$
\mean{B_{\eps}(x)} |Du|^p\, dy \leq c\mean{B_{2\eps}(x)}\left[\left|\frac{u-(u)_{B_{2\eps}(x)}}{\eps}\right|^p + a(y)\left|\frac{u-(u)_{B_{2\eps}(x)}}{\eps}\right|^q\right]\, dy\;. 
$$
In turn we estimate, for $y \in B_{\eps}(x)$ 
\begin{eqnarray*}
a(y)\left|\frac{u-(u)_{B_{2\eps}(x)}}{\eps}\right|^q &\stackleq{supfi}& c \left[\osc_{B_{2\eps}(x)}\, u\right]^{q-p}\eps^{p-q+\alpha}
\left|\frac{u-(u)_{B_{2\eps}(x)}}{\eps}\right|^p\\
&\stackleq{bound33}& c\eps^{(p-q)(1-\gamma)+\alpha}
\left|\frac{u-(u)_{B_{2\eps}(x)}}{\eps}\right|^p\\
&\stackleq{bound33}& c
\left|\frac{u-(u)_{B_{2\eps}(x)}}{\eps}\right|^p\;, 
\end{eqnarray*}
with $c$ being independent of $\eps$. 
Combining the last two inequalities gives \rif{cacc-lav} also in the case. In turn, by using the very definition of convolution, we can estimate 
\eqn{stimamorrey}
$$
|Du_\eps(x)| \leq  c\mean{B_{\eps}(x)} |Du| \, dy \leq c \left(\displaystyle{\mean{B_{\eps}(x)} |Du|^p} \, dy\right)^{1/p}
\stackleq{cacc-lav} \frac{c}{\eps^{1-\gamma}}\;.$$
In the last line we have used the fact that $u \in C^{0, \gamma}(\Omega)$ by assumption \rif{bound33} (when this is considered) and we have incorporated the case $u \in L^{\infty}$ and \rif{bound2} in the occurrence $\gamma=0$. Recalling the definition in \rif{defih0}, we continue to estimate 
\begin{eqnarray}
\nonumber
H\left(x, Du_\eps(x)\right) &\leq &[a(x)-a_{{\rm i}}(B_{2\eps}(x))]|Du_\eps(x)|^{q} + H_\eps(x,Du_\eps(x))\\
&\leq &c[a]_{0,\alpha}\eps^{\alpha}|Du_\eps(x)|^{q-p}|Du_\eps(x)|^p + H_\eps(x,Du_\eps(x))\notag \\
&\stackleq{stimamorrey} &c\eps^{\alpha+(q-p)(\gamma-1)}|Du_\eps(x)|^p + H_\eps(x,Du_\eps(x))\notag \\
&\stackleq{bound33}&  c|Du_\eps(x)|^p + H_\eps\left(x,Du_\eps(x) \right)\,.\notag
\end{eqnarray}
We therefore conclude that
\eqn{lastimaH}
$$
H\left(x, Du_\eps(x)\right) \leq  cH_\eps\left(x,Du_\eps(x) \right)\,, 
$$
where $c$ is independent of $\eps$. 
Finally, Jensen's inequality and \rif{defih0} yield
\begin{eqnarray*}
\nonumber H_\eps(x,Du_\eps(x)) &\leq &  \int_{B_{\eps}(x)}  H_\eps\left(x, Du(y)\right) \rho_\eps(x-y)\, dy\\
&\leq& \int_{B_{\eps}(x)} H(y,Du(y)) \rho_\eps(x-y)\, dy=[H(\cdot,Du(\cdot))*\rho_\eps](x)\,.
\end{eqnarray*}
The inequalities in the last two displays give that 
\eqn{comb}
$$
H\left(x, Du_\eps(x)\right) \leq  c[H(\cdot,Du(\cdot))*\rho_\eps](x)
$$
holds for every $x \in B$ and for a constant $c\equiv c (\data)$ which is independent of $\eps$.  

{\em $(p,q)$-phase.} This is when $a_{{\rm i}}(B_{2\eps}(x)) >  2[a]_{0, \alpha}\eps^\alpha$. By \rif{ult} we have that 
$
\|a(x)\|_{L^\infty(B_{2\eps}(x))} \leq 3 a_{{\rm i}}(B_{2\eps}(x)). 
$
This means that $H(y,z)\leq 3H_\eps(x,z)$ for every $y \in B_{2\eps}(x)$ and $z \in \er^n$, and in particular that $H(x,z)\leq 3H_\eps(x,z)$. We conclude again with \rif{lastimaH} and, proceeding as in the subsequent displays, we again arrive at \rif{comb}, which is established in every case/phase. 

With \rif{comb} at our disposal we are now ready to conclude the proof of \rif{convergenza1} and therefore of the whole Theorem \ref{main4}. We let $u_k := u_{\eps_{k}}$, where $\{\eps_{k}\} \subset (0,1/2)$ is a sequence such that $ \eps_k \to 0$ and $\eps_{k} \leq {\rm dist}(B, \partial \Omega)/20$ for every $k\in \en$. Obviously 
$u_k \in W^{1,\infty}(B)$ and $u_k \to u$ in $W^{1,p}(B)$ since $u \in W^{1,p}(B)$. Now, up to passing to non-relabelled subsequences, we may assume that 
$H\left(x, Du_k(x)\right)\to H\left(x, Du(x)\right)$ a.e. in $B$. 
Moreover, since $H(\cdot,Du(\cdot)) \in L^1(B)$, we have that $[H(\cdot,Du(\cdot))*\rho_{\eps_{k}}](x) \to H\left(x, Du(x)\right) $ in $L^1(B).$ Therefore, by \rif{comb} we can use a well-known variant of Lebesgue's dominated convergence theorem that implies that 
$H\left(x, Du_k(x)\right)\to H\left(x, Du(x)\right)$ in $L^1(B)$ that, in turn, gives \rif{convergenza1}. The proof is complete.

\end{document}